\documentclass[11pt]{article}
\usepackage[]{fontenc}
\usepackage[applemac]{inputenc}

\usepackage[a4paper]{geometry}
\usepackage{amsthm}
\usepackage{amsmath}
\usepackage{amssymb}
\usepackage{verbatim}
\usepackage{xcolor}
\usepackage[active]{srcltx}  % for the inverse search
\usepackage{palatino}

\usepackage[compact]{titlesec}
\let\oldbibliography\thebibliography
\renewcommand{\thebibliography}[1]{%
  \oldbibliography{#1}%
  \setlength{\itemsep}{0pt}%
  \footnotesize{}%
}

\makeatletter
\theoremstyle{plain}
\newtheorem{thm}{Theorem}[section]
  \theoremstyle{plain}
  \newtheorem{cor}[thm]{Corollary}
  \theoremstyle{plain}
  \newtheorem{lem}[thm]{Lemma}
  \theoremstyle{plain}
  \newtheorem{prop}[thm]{Proposition}
  \theoremstyle{plain}
  
  \theoremstyle{plain}
  \newtheorem{defn}[thm]{Definition}

\DeclareMathOperator{\supp}{supp}
\DeclareMathOperator{\diff}{diff}
\DeclareMathOperator{\essinf}{essinf}

\DeclareMathOperator{\diam}{diam}

\newcommand{\N}{\mathbb{N}}

\newcommand{\R}{\mathbb{R}}

\newcommand{\sqr}{{{}^{{}_\square}}}

\newcommand{\e}{\varepsilon}
\newcommand{\emdef}{\textbf}

\newcommand{\hide}[1]{}

\def\blfootnote{\xdef\@thefnmark{}\@footnotetext}

\makeatother

\begin{document}

\title{Equidistribution from fractal measures}

\author{Michael Hochman\thanks{Partially supported by ISF grant 1409/11 and ERC grant 306494.} $\,$ and Pablo Shmerkin\thanks{Supported by a Leverhulme Early Career Fellowship.}}

\maketitle
%AMS - for AMS comment out the line above and uncomment \maketitle below

\begin{abstract}We\blfootnote{\emph{2010 Mathematics Subject Classification} 11K16, 11A63, 28A80, 28D05.} give a fractal-geometric condition for a measure on $[0,1]$ to be supported on points $x$ that are normal in base $n$, i.e. such that $\{n^kx\}_{k\in\mathbb{N}}$ equidistributes modulo 1. This condition is robust under $C^1$ coordinate changes, and it applies also when $n$ is a Pisot number rather than an integer. As applications we obtain new results (and strengthen old ones) about the prevalence of normal numbers in fractal sets, and new results on measure rigidity, specifically completing Host's theorem to multiplicatively independent integers and proving a Rudolph-Johnson-type theorem for certain pairs of beta transformations.
\end{abstract}

%AMS\maketitle

%AMS\tableofcontents

\section{\label{sec:Introduction}Introduction}

\subsection{\label{seb:ackground}Background}

A number $x\in[0,1]$ is called $n$-normal, or normal in base $n$,
if $\{n^{k}x\}_{k\in\mathbb{N}}$ equidistributes modulo 1 for Lebesgue measure. This is the same as saying that the sequence of
digits in the base-$n$ expansion of $x$ has the same limiting statistics
as an i.i.d. sequence of digits with uniform marginals. It was E.
Borel who first showed that Lebesgue-a.e. $x$ is normal (in every base); thus the $n$-ary expansion of a typical number is maximally random. It is generally believed that, absent obvious obstructions, this phenomenon persists when it is relativised to ``naturally'' defined  subsets of the reals, i.e. that typical elements of well-structured sets, with respect to appropriate measures, are normal, unless the set displays an obvious obstruction. Taking this to the extreme and applying it to singletons one arrives at the folklore conjecture that natural constants such as  $\pi,e,\sqrt{2}$ are normal in every base. While the last conjecture seems very much out of reach of current methods, there are various positive results known for more substantial sets, often ``fractal'' sets. The present paper is a contribution in this direction.

It is better to work with measures than with sets, and it will be convenient to say that a measure $\mu$ is \emdef{pointwise $n$-normal} if it is supported on $n$-normal numbers. The first results on the problem above were obtained independently by Cassels
and W. Schmidt in the late 1950s \cite{Cassels59,Schmidt60}.
Motivated by a question of Steinhaus, who asked whether normality
in infinitely many bases implies it for all bases, they showed that the
Cantor-Lebesgue measure $\mu$ on the middle-$\frac{1}{3}$ Cantor
is pointwise $m$-normal whenever $m$ is not a power of
$3$. This answers Steinhaus's question negatively since no number
in the middle-$\frac{1}{3}$ Cantor set is $3$-normal.

The proofs of Cassels and Schmidt are analytical: they establish rapid decay, as $N\to\infty$, of the $L^{2}(\mu)$
norms of the trigonometric polynomials $\frac{1}{N}\sum_{k=0}^{N-1}e(mn^kt)$  appearing in Weyl's equidistribution
criterion (here and in what follows, $e(s)=\exp(2\pi is)$). An essentially sharp condition for pointwise $n$-normality in
terms of these norms was provided a few years later by Davenport,
Erd\H{o}s and LeVeque \cite{DEL63}. The latter theorem
underlies most subsequent work on the subject and is particularly effective when the measures are constructed with this method in mind, for example Riesz products, which are defined in terms of their Fourier transform. Many results have been obtained in this way by Brown, Pearce, Pollington, and Moran \cite{BMP85,BMP87,Pollington88,BMP93,Pollington95,MoranPollington95}. However, for most ``natural'' measures the required
norm bounds are nontrivial to obtain, if they can be obtained at all. They also are fragile in the sense that they do
not persist when the measure is perturbed. The book \cite{Bugeaud12} contains a thorough overview of many classical equidistribution results.

\subsection{\label{sub:A-condition-for-ae-equidistribution}Main results}

In this paper we give a new sufficient condition for pointwise $n$-normality, which is more dynamical and geometric in nature, and captures the spirit of the conjecture stated at the beginning of this introduction. Roughly speaking, we show that if the process of continuously magnifying the measure around a typical point does not exhibit any almost-periodic features at frequency $1/\log n$, then the measure is pointwise $n$-normal. While the condition is not a necessary one, it is a natural one
in many of the most interesting examples, and can be verified relatively easily  in many
cases where other methods fail. It also leads to many applications which we discuss below.

The condition is formulated in terms of an auxiliary measure-valued flow
which arises from the process of {}``zooming in'' on $\mu$-typical
points. This procedure has a long history, going back variously to Furstenberg \cite{Furstenberg70,Furstenberg08}, Z\"{a}hle \cite{Zahle88}, Bedford and Fisher \cite{BedfordFisher97}, M\"{o}rters and Preiss \cite{MortersPreiss98}, and Gavish \cite{Gavish11}; the following definitions are adapted from \cite{Hochman12b},
where further references can be found. Let $\mathcal{P}(X)$ denote
the space of Borel probability measures on a metric space $X$; when $X$ is compact we equip it with the Borel structure, and then the
space $\mathcal{P}(X)$ is then compact and metrizable
in the weak-{*} topology. Write $\mathcal{M}$ for the space of Radon (locally finite Borel) measures on $\mathbb{R}$ and $\supp\mu$ for the topological support of a measure $\mu\in\mathcal{M}$. Let
\[
\mathcal{M}^{\sqr}=\{\mu\in\mathcal{P}([-1,1])\,:\,0\in\supp\mu\}
\]
and for $\mu\in\mathcal{M}^\sqr$ and $t\in\mathbb{R}$, define $S_{t}\mu\in\mathcal{M}^\sqr$
by\footnote{In \cite{Hochman12b} $S_t$ was denoted $S_t^\sqr$ to emphasize that it was acting on $\mathcal{M}^\sqr$, and similarly in some of the later definitions, but this is not needed here and we drop the extra notation.}
\[
S_{t}\mu(E)=c\cdot\mu(e^{-t}E\cap[-1,1])
\]
where $c=c(\mu,t)$ is a normalizing constant. For $x\in\supp\mu$, similarly
define the translated measure by $\mu^{x}(E)=c'\cdot\mu((E+x)\cap[-1,1])$. The \emdef{scaling flow }is the Borel $\mathbb{R}^{+}$-flow
$S=(S_{t})_{t>0}$ acting on $\mathcal{M}^\sqr$. The \emdef{scenery}
of $\mu$ at $x\in\supp\mu$ is the orbit of $\mu^{x}$ under
$S$, that is, the one-parameter family of measures
$\mu_{x,t}=S_{t}(\mu^{x})$, $t\geq 0$.

Write $\mathcal{D}=\mathcal{P}(\mathcal{P}([-1,1]))$, which is again
compact and metrizable and $\mathcal{P}(\mathcal{M}^\sqr)\subseteq\mathcal{D}$.\footnote{We would have liked to define $\mathcal{D}=\mathcal{P}(\mathcal{\mathcal{M}^\sqr})$, but while  $\mathcal{M}^\sqr$ is a Borel set it is not topologically nice. This is why we define  $\mathcal{D}$ as above, and why the test functions  in the definition of equidistribution in $\mathcal{D}$ are taken from $C(\mathcal{P}([-1,1])$ and not from $C(\mathcal{M}^\sqr)$.} For clarity we refer to elements of $\mathcal{D}$
as \emdef{distributions}, whereas we continue to refer to the elements
of $\mathcal{M}^\sqr$ as measures. A measure $\mu\in\mathcal{P}(\mathbb{R})$
\emdef{generates a distribution $P\in\mathcal{D}$ at $x\in\supp\mu$} if
the scenery at $x$ equidistributes for $P$ in $\mathcal{D}$, i.e.
if
\[
\lim_{T\rightarrow\infty}\frac{1}{T}\int_{0}^{T}f(\mu_{x,t})\, dt=\int f(\nu)\, dP(\nu)\qquad\mbox{for all }f\in C(\mathcal{P}([-1,1])),
\]
and $\mu$ generates $P$ if it generates $P$ at $\mu$-a.e. $x$.

If $\mu$ generates $P$, then $P$ is supported on $\mathcal{M}^\sqr$ and $S$-invariant (while unsurprising this is not completely trivial since $S$ acts discontinuously, see \cite[Theorem 1.7]{Hochman12b} for the proof). We say that
$P$ is \emdef{trivial} if it is the distribution supported on the measure
$\delta_{0}\in\mathcal{M}^\sqr$, which is a fixed point of $S$. It
can be shown that if $\mu$ generates a distribution, then it is the trivial one if and
only if $\mu$ gives full mass to a set of zero Hausdorff dimension (this follows from \cite[Proposition 1.19]{Hochman12b}).

To an $S$-invariant distribution $P$ we associate its \emdef{pure-point spectrum} $\Sigma(P,S)$. This is the set of $\alpha\in\mathbb{R}$ for which
there exists a non-zero measurable function $\varphi:\mathcal{M}^\sqr\rightarrow\mathbb{C}$
satisfying $\varphi\circ S_{t}=e(\alpha t)\varphi$, $t\in\R$, on a set
of full $P$-measure. The existence of such an eigenfunction indicates that some non-trivial feature of the measures of $P$ repeats periodically when the measures are magnified by a factor of $e^\alpha$.

Finally, let $f\mu$ denote the push-forward of the measure $\mu$, i.e. $(f\mu)(A)=\mu(f^{-1}A)$. We note this is sometimes denoted $f_{\#}\mu$.

\begin{thm}
\label{thm:WM-case}Let $\mu\in\mathcal{M}$ be a measure generating
a non-trivial $S$-ergodic distribution $P\in\mathcal{D}$, and let
$n\in\mathbb{N}$, $n\ge 2$. If $\Sigma(P,S)$ does not contain a non-zero integer multiple of
$1/\log n$, then $\mu$ is pointwise $n$-normal. Furthermore,
the same is true for $f\mu$ for all $f\in\diff^{1}(\mathbb{R})$.
\end{thm}

The non-triviality assumption means that the theorem does not apply to measures supported on zero-dimensional sets. This limitation is intrinsic to our methods.

The hypotheses of the theorem may seem restrictive, since general measures do not generate any distribution, let alone an ergodic one satisfying the spectral condition. However, ``natural'' measures arising in dynamics, fractal geometry or arithmetic, very often do generate an $S$-ergodic distribution (see e.g. \cite{Gavish11, Hochman12, Hochman12b} for many examples), and the important hypothesis becomes the spectral one. It is possible to formulate a version of the theorem that applies to measures which do not generate a distribution in the above sense, but the result is less useful. See remark at the end of Section \ref{sub:proof-of-WM-case}. In Section \ref{sec:refinements} we give some stronger versions of the theorem which are used in some of the later applications.

Finally, note that the theorem is not a characterization, and the presence of $k/\log n$ in the pure point spectrum of $P$ does not rule out pointwise $n$-normality. Indeed, if a measure is translated by a random, uniformly chosen distance, then the sceneries are not affected, but almost surely the measure becomes pointwise normal in every base (see also Theorem \ref{thm:analytic-images-of-ssms} below). It is worth mentioning though that the canonical example of a measure that is \textit{not} pointwise $n$-normal is that of a singular measure on $[0,1]$ invariant and ergodic for $x\mapsto nx\bmod 1$. For such $\mu$,  the first author showed in \cite{Hochman12} that, when the entropy is positive, the generated distribution indeed has a multiple of $1/\log n$ in its spectrum.

There is some interest also in expansions of numbers in non-integer bases. Following R\'{e}nyi \cite{Renyi57}, for $\beta>1$ we define the $\beta$-expansion of $x\in [0,1)$ to be the lexicographically least sequence $x_n\in\{0,1,\ldots,\lceil\beta-1\rceil\}$ such that $x=\sum_{n=1}^\infty x_n\beta^{-n}$. This sequence is obtained from the orbit of $x$ under $T_\beta:x\mapsto\beta x\bmod 1$ in a manner similar to the integer case. It is known that $T_\beta$ has a unique absolutely continuous invariant measure, called the Parry measure, and we shall say that $x$ is $\beta$-normal if under $T_\beta$ it equidistributes for this measure.

Recall that $\beta>1$ is called a Pisot number if it is an algebraic integer whose algebraic conjugates are of modulus strictly smaller than $1$. We adopt the convention that integers $\ge 2$ are Pisot numbers. The dynamics of $T_\beta$ is best understood for this class of numbers, and our results extend to them:

\begin{thm}\label{thm:WM-pisot-case}
Theorem \ref{thm:WM-case} holds as stated for a Pisot  number $\beta>1$ in place of $n$.
\end{thm}

It is possible that the Pisot assumption is unnecessary but currently we are unable to prove this (but see also the discussion following Corollary \ref{cor:rudolph-for-beta} below). On the other hand, Bertrand-Mathis \cite{Bertrand-Mathis79} proved that if $\beta$ is Pisot and $x$ is $\beta$-normal, then $\{ \beta^n x\}_{n=1}^\infty$ equidistributes on the circle. Hence we have:

\begin{cor}\label{cor:bertrand-mathis-cor}
If $\beta>1$ is Pisot and $\mu$ satisfies the hypothesis of the Theorem  \ref{thm:WM-case} with $\beta$ in place of $n$, then $\{\beta^nx\}_{n=1}^\infty$ equidistributes modulo 1 for $\mu$-a.e. $x$.
\end{cor}

Before turning to applications let us say a few words about what goes into the proof of Theorem \ref{thm:WM-pisot-case} (a more detailed sketch of the proof is given in Section \ref{sub:sketch-of-proof}). There are two main ingredients. The first involves the behavior of the dimension of measure under convolution. Specifically, among the measures of positive dimension invariant under $x\mapsto \beta x\bmod 1$, one can characterize Lebesgue measure (or the Parry measure) in terms of its dimension growth under convolutions. This part of the argument is special to the dynamics of $x\mapsto \beta x\bmod 1$ and is the main place where the Pisot property is used in the non-integer case. Most of the work then goes into showing that, if there were a measure $\mu$ satisfying the hypothesis of the theorems above but not their conclusion, then one could concoct an invariant measure  $\eta$  violating the characterization alluded to above. The scheme above is a refinement of ideas we have used before in \cite{HochmanShmerkin12} and \cite{Hochman12b}.

The second ingredient in the proof, and one of the main innovations in this paper, applies in a more general setting than invariant measures for piecewise-affine maps of  $[0,1]$. The proper context is that of a Borel map $T$ of a compact metric space $X$. Roughly speaking, we show how to relate the small-scale structure of a measure $\mu$ on $X$ to the distribution of $T$-orbits of $\mu$-typical points. This result, while classical in nature, appears to be new and we believe it may find further applications. We leave the discussion and precise statement to Section \ref{sub:distrib-of-orbits}.

\subsection{\label{sub:apps}Applications}

\subsubsection{\label{sub:normals-nums-in-fractals}Normal numbers in fractals}

As our first application we consider sets arising as attractors of iterated  function systems, or, equivalently, repellers of uniformly expanding maps on the line (see below for definitions). We show that, under some weak regularity assumptions, if such a set is defined by nonlinear dynamics, or if the contraction rates of the defining maps satisfy a natural algebraic condition, then typical points in the set are $n$-normal. This should be interpreted in terms of the conjecture stated earlier: indeed, it implies that if such a set contains no $n$-normal numbers, then the set is essentially defined by linear\footnote{We shall follow the convenient but imprecise convention of using the term linear also for affine maps.} maps whose slopes are rational powers of $n$, and in this sense the dynamics is similar to the  canonical examples of sets without $n$-normal numbers, namely closed subsets of $[0,1]$ that are invariant under the piecewise-linear maps $x\mapsto nx\bmod 1$.

We start with the relevant definitions.  An \emdef{iterated function system} (IFS) is a finite family $\mathcal{I}=\{f_{0},\ldots,f_{r-1}\}$
of strictly contracting maps $f_{i}:I\rightarrow I$ for a compact interval $I\subseteq\mathbb{R}$ (of course one can define IFSs in general metric spaces). The IFS is of class $C^\alpha$ if all the $f_i$ are. We shall say that the IFS $\mathcal{I}$ is \emdef{regular} if the maps $f_i$ are orientation-preserving injections, and the intervals $f_i(I)$ are disjoint except possibly at their endpoints so, in particular, the so-called open set condition is satisfied. In this article we will only consider $C^{1+\e}$ regular IFSs, but some of the assumptions can be relaxed. For example, the orientation-preserving assumption is just for simplicity and can be easily dropped.

The \emdef{attractor}\footnote{There is an equivalent dynamical description of attractors of $C^\alpha$-IFSs, namely, as the maximal compact invariant sets of expanding $C^\alpha$ maps $I\mapsto\mathbb{R}$.} of $\mathcal{I}$ is the
unique nonempty compact set $X\subseteq I$ satisfying \[
  X=\bigcup_{i\in [r]} f_{i}(X)
\]
(here and throughout the paper, $[r]=\{0,\ldots,r-1\}$). There are a number of natural measures one can place on $X$. One is the $\dim X$-dimensional Hausdorff measure, which for a $C^{1+\e}$-IFS is positive and finite on $X$. Another good class are the \emdef{self-conformal measures} (also called \emdef{self-similar measures} if the maps $f_i$ are linear), that is, measures satisfying the relation \[
  \mu=\sum_{i\in [r]} p_i \cdot f_i\mu
\]
for a positive probability vector $(p_0,\ldots,p_{r-1})$. Both of the examples above  are special cases of \emdef{Gibbs measures} for
H\"{o}lder potentials $\varphi:X\to\mathbb{R}$. We will not define Gibbs measures, but rather rely on a standard property of such measures $\mu$, namely,  that there is a constant
$C>1$ such that for all finite sequences $i_1,\ldots,i_k$, $j_1,\ldots,j_\ell\in [r]$,
\begin{equation}
C^{-1} \le \frac{\mu(f_{i_1}\cdots f_{i_k}I)\mu(f_{j_1}\cdots f_{j_\ell}I)}{\mu(f_{i_1}\cdots f_{i_k}f_{j_1}\cdots f_{j_\ell}I)}\le C. \label{eq:quasi-product}
\end{equation}
We shall call measures satisfying this property \emdef{quasi-product measures} (or quasi-Bernoulli measures). This is a broader class than Gibbs measures for H\"{o}lder potentials; for example, it contains Gibbs measures for almost-additive sequences of potentials, see \cite{Barreira06}.

Our first result assumes an algebraic condition on the contractions. For a $C^1$-contraction $f$ on $\mathbb{R}$, we define its \emdef{(asymptotic) contraction ratio} to be  $\lambda(f)=f'(p)$, where $p$ is the unique fixed point of $f$. For affine $f$ this is just the usual contraction ratio; to justify the name in the nonlinear case note that for every distinct pair of points $x,y$, \[
    \lambda(f)=\lim_{n\to\infty} -\frac{\log |f^n(x)-f^n(y)|}{n}.
\]

Write $a\sim b$ if $a,b$ are integer powers of a common number, equivalently $\log a/\log b\in\mathbb{Q}$; otherwise write $a\not\sim b$, in which case $a,b$ are said to be \emdef{multiplicatively independent}.

\begin{thm} \label{thm:dissonant-IFSs}
Let $\mathcal{I}$ be a $C^{1+\varepsilon}$ IFS that is regular in the sense above, and $\beta>1$ a Pisot number.\footnote{We remark again that in this and subsequent statements, the Pisot assumption includes the possibility $\beta\in\N$. } If there exists an $f\in\mathcal{I}$ with  $\lambda(f)\nsim\beta$, then any quasi-product measure $\mu$ for $\mathcal{I}$ is pointwise $\beta$-normal, and so is $g\mu$ for all $g\in\diff^1(\R)$.
\end{thm}

The classical results of Cassels and Schmidt are special cases of this for certain IFSs consisting of affine maps with the same contraction ratio. We note that the result above is new even when the IFS is affine and contains maps with two multiplicatively independent contraction ratios; classical methods break down since nothing seems to be known about the decay (or lack thereof) of the Fourier transform of natural measures on such attractors.

Our second result says that nonlinearity and enough regularity are sufficient for pointwise normality, irrespective of algebraic considerations. More precisely, we say an IFS $\mathcal{I}=\{f_i\}$ is \emdef{linear} if all of the maps in $\mathcal{I}$ are affine maps, and \emdef{non-linear} otherwise. We say that $\mathcal{I}$ is \emdef{totally non-linear} if it is not conjugate to a linear IFS via a $C^1$ map; here an IFS is $\mathcal{J}$ is $C^\alpha$-conjugate to $\mathcal{I}$ if it has the form  $\mathcal{J}=g\mathcal{I}=\{gf_ig^{-1}\}$ for a $C^\alpha$-diffeomorphism $g$.

\begin{thm} \label{thm:totally-nonlinear-IFSs}
Let $\mathcal{I}$ be a $C^{\omega}$ IFS that is regular in the sense above and $\beta>1$ a Pisot number. If $\mathcal{I}$ is totally non-linear, then any quasi-product measure $\mu$ for $\mathcal{I}$ is pointwise $\beta$-normal, and so is $g\mu$ for all $g\in\diff^1(\R)$.
\end{thm}

The two theorems above have substantial overlap and each of them is generic in the appropriate space of IFSs. The algebraic condition is generally  the easier one to verify, and the regularity assumptions are weaker, though it seems very probable that weaker regularity assumptions are sufficient in the totally non-linear case also.

It also seems very likely that non-linearity, rather than total non-linearity, should suffice in Theorem \ref{thm:totally-nonlinear-IFSs}. We are able to prove such a result for a smaller class of measures. Namely,

\begin{thm}\label{thm:analytic-nonlinear-IFSs}
Let $\mathcal{I}$ be a $C^\omega$ IFS that is regular in the sense above, and $\beta>1$ a Pisot number.
If $\mathcal{I}$ is non-linear, then every self-conformal measure for $\mathcal{I}$ is pointwise $\beta$-normal.
\end{thm}

In this theorem, the totally non-linear case is covered by Theorem \ref{thm:totally-nonlinear-IFSs}. In the conjugate-to-linear case, if $g\in\diff^\omega(\R)$ conjugates $\mathcal{I}$ to a linear IFS $\mathcal{J}=g\mathcal{I}$, then $\mu=g^{-1}\nu$ where $\nu$ is a self-similar measure for $\mathcal{J}$, and $g$ is not affine (since $\mathcal{J}$ is linear and $\mathcal{I}=g^{-1}\mathcal{J}$ is not). Also, it is a remarkable consequence of the work of Sullivan \cite{Sullivan88} and Bedford-Fisher \cite{BedfordFisher97} that if a  $C^{\alpha}$-IFSs, $\alpha>1$, is $C^1$-conjugate to a linear IFS then it is also $C^\alpha$-conjugate to a linear IFS (see \cite[Theorem 7.5]{BedfordFisher97}). Thus, Theorem \ref{thm:analytic-nonlinear-IFSs} follows from the following one:

\begin{thm}\label{thm:analytic-images-of-ssms} Let $\mu$ be a self-similar measure for a linear IFS that is regular in the sense above. Then for any non-affine real-analytic $g\in\diff^\omega(\mathbb{R})$, $g\mu$ is pointwise $\beta$-normal for every Pisot $\beta>1$.\end{thm}

Here is one concrete consequence of the results above.

\begin{cor} Let $\mu$ denote the Cantor-Lebesgue measure on the middle-1/3 Cantor set. Then $x^2$ is $3$-normal for  $\mu$-a.e. $x$.
\end{cor}

The point is, of course, that no points in the middle-1/3 Cantor set are $3$-normal themselves.

The corollary above is immediate from the previous theorem and the use of the square function is incidental. In fact for we could replace $x^2$ with $f(x)$ for $f\in\diff^2$. From Theorem \ref{thm:dissonant-IFSs} we can reduce the regularity to $\diff^1$ if we only want $n$-normality for $n\nsim 3$. These differences perhaps indicate that our regularity assumptions may be suboptimal. Note that for $f=$ identity, this again is the theorem of Cassels and Schmidt, and their spectral methods carry over to translations, but the stability under perturbation is new even for affine $f$. Related to this question, we note that Bugeaud, Fishman, Kleinbock and Weiss \cite{BBFKW10} have shown that for many fractals sets, including self-similar sets satisfying the open set condition,
there is a full-dimension subset consisting of numbers which are \textit{not} normal
in any integer base. Moreover their result holds for any bi-Lipschitz
image of the set. The stability of our results under bi-Lipschitz transformations remains open.

While this paper was in revision we learned of Kaufman's paper \cite{Kaufman1984}. Kaufman studies differentiable images of certain Bernoulli convolutions, and obtains polynomial decay of the
Fourier transform of their image under $C^2$ diffeomorphisms, implying in particular pointwise normality of the images. His results apply to linear self-similar measures defined by two maps with the same contraction ratio, and equal weights (it is likely the method can be adapted to more than two maps, but unlikely that the equicontraction assumption can be dropped with current methods). In particular the last corollary follows from Kaufman's work.

\subsubsection{\label{sub:host}Host's theorem and measure rigidity}

Let $n\in\mathbb{N}$ and let $T_n:[0,1]\to [0,1]$ denote the map $T_nx=nx\bmod 1$. An important phenomenon concerning these maps is measure rigidity: a well-known conjecture of  Furstenberg states that, if $m\not\sim n$, then the only probability measures jointly invariant under $T_m$ and $T_n$ are combinations of Lebesgue measure and atomic measures on rational points. This conjecture, known as the times-2, times-3 conjecture, is the prototype for many similar conjectures in other contexts, see e.g. \cite{Lindenstrauss05}. The best result towards it is due to  Rudolph and Johnson \cite{Rudolph90, Johnson92}: if  a measure has positive entropy and is jointly invariant and ergodic under $T_m,T_n$ for $m\not\sim n$, then it is Lebesgue. Although nothing is known about the zero-entropy case, in the positive entropy case there is a pointwise strengthening of the  Rudolph-Johnson theorem for $\gcd(m,n)=1$, due to B. Host \cite[Th\'{e}or\`{e}me 1]{Host95}:

\begin{thm}Let $m,n\geq 2$ be integers and $\gcd(m,n)=1$. Suppose $\mu$ is an invariant and ergodic measure for $T_{n}$ of positive entropy. Then $\mu$ is pointwise $m$-normal.
\end{thm}

This implies the Rudolph-Johnson theorem in the case $\gcd(m,n)=1$: if $\mu$ is a jointly $T_m,T_n$ invariant measure and all $T_n$ ergodic components have positive entropy, then by the theorem $\mu$-a.e. point equidistributes for Lebesgue under $T_m$. But by the ergodic theorem, it also equidistributes for the ergodic component of $\mu$ to which it belongs; hence $\mu$ is Lebesgue.

The hypothesis of Host's theorem, however, is stronger than it ``should'' be, i.e. it is stronger than the hypothesis of the Rudolph-Jonson Theorem. Lindenstrauss \cite{Lindenstrauss01} showed that the conclusion holds under the weaker assumption that $n$ does not divide any power of $m$, but this is still too strong.\footnote{Host's theorem has also been generalized in some other directions, see Meiri \cite{Meiri1998}} On the other hand, Feldman and Smorodinsky \cite{FeldmanSmorodinsky92} had earlier proved a similar result assuming only that $m\nsim n$, but under the strong assumption that the measure $\mu$ is weak Bernoulli. In that work it is conjectured that the same holds assuming only that $\mu$ is ergodic and has positive entropy. The following theorem gives the result in its ``correct'' generality and for some non-integer bases, and also shows that it is stable under smooth enough perturbation.

\begin{thm} \label{thm:application-HostLindenstrauss}
Let $\beta,\gamma>1$ with $\beta$ a Pisot number, and $\beta\not\sim\gamma$. Then any $T_\gamma$-invariant and ergodic measure $\mu$ with positive entropy is pointwise $\beta$-normal. Furthermore the same remains true for $g\mu$ for any $g\in\diff^2(\mathbb{R})$.
\end{thm}

Of course, the same is true under the assumption that all $T_\gamma$-ergodic components of $\mu$ have positive entropy. Note the asymmetry in the requirement from $\beta,\gamma$. We do not know whether the Pisot assumption is unnecessary, but we note that Bertrand-Mathis \cite{Bertrand-Mathis77} has obtained some complementary results for $\gamma$ Pisot and $\beta$ arbitrary, though only for measures that satisfy the weak-Bernoulli property with respect to the natural symbolic coding of $T_\gamma$.

From this one derives a new measure rigidity result for $\beta$-maps.

\begin{cor}\label{cor:rudolph-for-beta}
Let $\beta,\gamma>1$ with $\beta\not\sim\gamma$ and $\beta$ Pisot. If $\mu$ is jointly invariant under $T_{\beta},T_{\gamma}$, and if all ergodic components of $\mu$ under $T_\gamma$ have positive entropy, then $\mu$ is the common Parry measure for $\beta$ and $\gamma$; in particular, it is absolutely continuous. The same holds if $T_{\beta},T_{\gamma}$ are conjugated separately by $C^{2}$-diffeomorphisms.
\end{cor}
\begin{proof}
If $\mu$ is as in the statement, then by Theorem \ref{thm:application-HostLindenstrauss}, $\mu$-almost all $x$ equidistribute under $T_\beta$ for the $\beta$-Parry measure (i.e. an absolutely continuous measure). On the other hand, by the ergodic theorem $\mu$-a.e. $x$ equidistributes for the $T_\gamma$-ergodic component to which it belongs; hence $\mu$ is also the Parry measure for $T_\gamma$.

The latter assertion follows in the same way, using that $g\mu$ is pointwise $\beta$-normal for all $g\in\diff^2(\R)$.
\end{proof}

We hope to be able to eliminate the Pisot assumption in this result; this will be addressed in a forthcoming paper. Corollary \ref{cor:rudolph-for-beta} also improves \cite[Corollary 1.5]{Hochman12} by eliminating the ergodicity assumption. We do not know for what pairs $(\beta,\gamma)$ the Parry measures coincide, or even whether this may happen for different non-integer $\beta,\gamma$.

\subsubsection{\label{sub:BA-normal-numbers}Badly approximable normal numbers}

Another application concerns continued fraction representations and their relation to integer expansions. Let $\Lambda\subseteq\mathbb{N}$ be a finite set with at least two elements, and set
\[
C_{\Lambda}=\{x\in[0,1]\,:\, x\mbox{ has only symbols from }\Lambda\mbox{ in its continued fraction expansion}\}.
\]
These sets are natural in Diophantine approximation since their union over all finite $\Lambda\subseteq\mathbb{N}$ is the set of badly approximable numbers. The question of whether there are badly approximable normal numbers reduces to asking whether any of the $C_\Lambda$ contain normal numbers. An affirmative answer follows from work of Kaufman \cite{Kaufman80}, who, assuming  $\dim C_\Lambda>2/3$, constructed probability
measures on $C_{\Lambda}$ whose Fourier transform decays polynomially. The bound on the dimension was relaxed to $\dim C_\Lambda>1/2$ by Queff{\'e}lec and Ramar{\'e} \cite{QueffelecRamare03}. Thus, for example, there are normal numbers whose continued fraction expansions consist only of the digits $1,2$ (because $\dim C_{\{1,2\}}>1/2$). However, the methods from those papers fail below  dimension $1/2$, so, for example, it was not known whether there are normal numbers with continued fraction coefficients $5,6$.

We note that $C_\Lambda$ is the attractor of a regular IFS, namely $\{ f_i\circ f_j:i,j\in\Lambda\}$, where $\{ f_i\}$ are the inverse branches of the Gauss map (the reason for the compositions is that, although $f_1$ is not a strict contraction, all the compositions $f_i\circ f_j$ are). As an application of Theorem \ref{thm:dissonant-IFSs}, we have:

\begin{thm}\label{thm:BA-normal-numbers}
Any quasi-product measure on $C_\Lambda$ (in particular the $\dim C_\Lambda$-dimensional Hausdorff measure) is pointwise $\beta$-normal for any Pisot $\beta>1$.
\end{thm}

Even when $\dim C_\Lambda>1/2$, this improves the results of Kaufman, Queff{\'e}lec and Ramar{\'e}, in that the result holds for a broader and more natural class of measures. The result on normality in non-integer Pisot bases is new in all cases. It seems very likely that the result holds also for Gibbs measures when $\Lambda\subseteq\mathbb{N}$ is infinite,  under standard assumptions on the Gibbs potential, but we do not pursue this.

One natural question is whether a reciprocal of Theorem \ref{thm:BA-normal-numbers} holds. For example, is it true that almost all points in the middle-$1/3$ Cantor set are normal with respect to the Gauss map $G$? (i.e. they equidistribute under $G$ for the Gauss measure, which is the only absolutely continuous $G$-invariant measure). To the best of our knowledge, it is not even known whether there exists a point which is Gauss normal but not $n$-normal for any $n$ (in the positive direction, Einsiedler, Fishman and Shapira \cite{EFS11} recently proved that almost all points in the middle-$1/3$ Cantor set have unbounded partial quotients, i.e. are not contained in any $C_\Lambda$). Unfortunately, our methods do not seem to help with this problem. The (piecewise) linearity of $T_\beta$ is strongly used in the part of the proof that deals with the geometric behavior of invariant measures under convolution. In particular, Theorem \ref{thm:existence-of-resonant-measures} seems to fail for the Gauss map and likely for most non-linear and many piecewise linear maps.

\subsection{Organization of the paper}
In the next section we state and prove a general result relating orbits of $\mu$-typical points to the structure of $\mu$; this is the second main component of the proof of Theorem \ref{thm:WM-pisot-case} referred to above. Section \ref{sec:prlim-on-dim} collects some background on dimension. In section \ref{sec:efds-spectra-and-phase} we recall some background on the pure point spectrum and eigenfunctions of flows, and discuss the class of distributions arising from scenery flows, called ergodic fractal distributions. We also introduce the concept of phase measure and its main properties. We prove Theorem \ref{thm:WM-pisot-case} in Section \ref{sec:proof-of-WM-case} (with a key component postponed to Section \ref{sec:resonance}). In Section \ref{sec:app-to-fractals-1} we derive Theorems \ref{thm:dissonant-IFSs}, \ref{thm:totally-nonlinear-IFSs} and \ref{thm:BA-normal-numbers}. Finally, in Section \ref{sec:refinements} we prove some variants of Theorem \ref{thm:WM-case}, and employ them to prove Theorems \ref{thm:application-HostLindenstrauss}  and \ref{thm:analytic-images-of-ssms}.

\subsection*{Acknowledgment}
Part of this work was carried out while M.H. was visiting at the Theory group at Microsoft Research (Redmond); many thanks to the members of the group for their hospitality and support.

\section{\label{sub:distrib-of-orbits}Relating the distribution of orbits to the measure}

While most of our considerations in this paper are special to $\mathbb{R}$, those in this section apply in the following very general setting. Let $X$ be a compact metric space and $T:X\to X$ a Borel measurable map.\footnote{Compactness is only required in order to define weak-* convergence (i.e. provide a natural algebra of test function), but the core of the discussion below is purely measure-theoretic.} For Borel probability measures $\mu,\nu$ on $X$, let us say that a measure $\mu$ is \emdef{pointwise generic} for $\nu$ if $\mu$-a.e. $x$ equidistributes for $\nu$ under $T$, that is, \begin{equation}
      \frac{1}{N}\sum_{n=0}^{N-1}f(T^nx)\to\int f\,d\nu \;\;\;\;\mbox{ for every } f\in C(X). \label{eq:equidistribution}
\end{equation}

This notion appears in many contexts, although the name is not standard. Clearly when $Tx=nx\bmod 1$ and $\nu$ is Lebesgue measure on $[0,1]$, this is the same as pointwise  $n$-normality. A well-known variant appears in smooth dynamics: when $X$ is a manifold, a measure $\nu$ is called the Sinai-Ruelle-Bowen (SRB) measure if the volume measure on $X$ is pointwise generic for $\nu$. Other examples include the study of badly approximable points on analytic curves in $\mathbb{R}^d$, and similar applications in arithmetic contexts.

While one does not expect to be able to say very much for arbitrary maps and measures, there is an obvious formal strategy to follow if one wants to prove that $\mu$ is pointwise generic for $\nu$: it is sufficient to show that for $\mu$-a.e. $x$, if $x$ equidistributes for a measure $\eta$ along some subsequence of times (i.e. \eqref{eq:equidistribution} holds along some $N_k\to\infty$), then $\eta=\nu$.

To go any further with this scheme, one needs a way to relate measures $\eta$ arising as above to the original measure $\mu$. It is not obvious that such a relation exists: $\eta$ is determined primarily by the point $x$, and although $x$ is $\mu$-typical, once it is selected, it would appear that the role of $\mu$ has ended. However, it turns out that there is a very close connection between $\eta$ and $\mu$, provided by the theorem below. Roughly speaking, it shows that, under a mild technical condition, one can express $\eta$ as a weak limit of ``pieces'' of $\mu$, ``magnified'' via the dynamics.

For a finite measurable partition $\mathcal{A}$ of $X$, write $T^{i}\mathcal{A}=\{T^{-i}A\,:\, A\in\mathcal{A}\}$
and $\mathcal{A}^{n}=\bigvee_{i=0}^{n}T^{i}\mathcal{A}$ for the coarsest
common refinement of $\mathcal{A},T\mathcal{A},\ldots,T^{n}\mathcal{A}$. Also
let $\mathcal{A}^{\infty}=\bigvee_{i=0}^{\infty}T^{i}\mathcal{A}$
denote the $\sigma$-algebra generated by the partitions $\mathcal{A}^n$,
$n\geq0$. We say that  $\mathcal{A}$ is a \emdef{generator} for $T$ if $\mathcal{A}^{\infty}$
is the full Borel algebra. When $T$ is invertible, we similarly define $\mathcal{A}^{\pm n}=\bigvee_{i=-n}^n T^i\mathcal{A}$ and $\mathcal{A}^{\pm\infty}=\bigvee_{i=-\infty}^{\infty}T^{i}\mathcal{A}$, and say that $\mathcal{A}$ is a generator if $\mathcal{A}^{\pm\infty}$ is the full Borel algebra. Finally, we say that  $\mathcal{A}$ is a \emdef{topological generator} if $\sup\{\diam A \,:\,A\in\mathcal{A}^n\} \to 0$ as $n\to\infty$ (or, in the invertible case, the sup is over $A\in\mathcal{A}^{\pm n}$). A topological generator is clearly a generator.

Write $\mathcal{A}(x)\in\mathcal{A}$ for the unique element $A\in\mathcal{A}$
containing $x$. Given $\mu\in\mathcal{P}(X)$ and a point $x\in X$
such that $\mu(\mathcal{A}^{n}(x))>0$, let
\[
\mu_{\mathcal{A}^{n}(x)}=c\cdot T^{n}(\mu|_{\mathcal{A}^{n}(x)})
\]
where $c=\mu(\mathcal{A}^{n}(x))^{-1}$ is a normalizing constant.
For a.e. $x$, this is well-defined for all $n$.

\begin{thm}\label{thm:local-average}
Let $T:X\rightarrow X$ be a Borel-measurable
map of a compact metric space, $\mu$ be a Borel probability measure
on $X$ and $\mathcal{A}$ a generating partition.
Then for $\mu$-a.e. $x$, if $x$ equidistributes for $\nu\in\mathcal{P}(X)$
along some $N_{k}\rightarrow\infty$, and if
$\nu(\partial A)=0$ for all $A\in\mathcal{A}^{n}, n\in\N$, then
\begin{equation}
\nu=\lim_{k\rightarrow\infty}\frac{1}{N_{k}}\sum_{n=1}^{N_{k}}\mu_{\mathcal{A}^{n}(x)}\qquad\mbox{weak-* in }\mathcal{P}(X). \label{eq:local-average}
\end{equation}
If, furthermore, $\mathcal{A}$ is a topological generator, then the hypothesis on $\nu$ follows if, for all $m$, \begin{equation}
   \limsup_{k\rightarrow\infty}\frac{1}{N_{k}}\sum_{n=1}^{N_{k}}\mu_{\mathcal{A}^{n}(x)}(C_m^{(\varepsilon)})\;=\;o(1) \qquad \mbox{as }\varepsilon\to 0, \label{eq:boundary-condition}
\end{equation}
where $C_m=\bigcup_{A\in\mathcal{A}^m}\partial A$, and $C_m^{(\varepsilon)}$ is its  $\varepsilon$-neighborhood.
\end{thm}

Note that if in the right hand side of \eqref{eq:local-average} we replace $\mu_{\mathcal{A}^{n}(x)}$
by $\delta_{T^{n}x}$, then the convergence to $\nu$ is just a reformulation
of the definition of equidistribution. Generally $\mu_{\mathcal{A}^{n}(x)}$
and $\delta_{T^{n}x}$ are very different measures and the content of the
theorem is that these two sequences are nevertheless asymptotic in
the Ces\`{a}ro sense. This is quite surprising, and such a general fact
can only be due to very general principles, as we shall see in the proof.

\begin{proof}
We give the proof assuming that $\mathcal{A}$ is forward generating and comment
on the invertible case at the end.

Let $\mathcal{F}$ denote the set of
linear combinations of indicator functions of $A\in\mathcal{A}^{n}$,
$n\in\mathbb{N}$, with coefficients in $\mathbb{Q}$. This is a countable
algebra and, for $x\in X$ and $\nu\in\mathcal{P}(X)$ such that $\nu(\partial A)=0$ for $A\in\mathcal{A}^{n}$, it is well known that $x$ equidistributes for $\nu$ along $N_{i}$ if and  only if $\lim\frac{1}{N_{i}}\sum_{n=1}^{N_{i}}f(T^{n}x) = \int f\, d\nu$ for every $f\in\mathcal{F}$ (it is here that we use the assumption that $\nu$ gives zero mass to the boundaries of $A\in\mathcal{A}^n$). Similarly, the limit in the conclusion of the theorem holds if and only if $\lim\frac{1}{N_{i}}\sum_{n=1}^{N_{i}}\int f(x)\,d\mu_{\mathcal{A}^n(x)} = \int f\, d\nu$ for all $f\in\mathcal{F}$.
It follows, then, that to prove the theorem it suffices for us to
show that for $\mu$-a.e. $x$,
\begin{equation}
\lim_{N\rightarrow\infty}\frac{1}{N}\sum_{n=0}^{N-1}\left(\int f\,d\mu_{\mathcal{A}^{n}(x)}-f(T^{n}x)\right)\rightarrow0\qquad\mbox{for every }f\in\mathcal{F}. \label{eq:martingale-difference}
\end{equation}

Suppose that $f=\sum a_{i}1_{A_{i}}$ where $a_{i}\in\mathbb{Q}$
and $A_{i}\in\mathcal{A}^{k}$ for some $k$. Notice that by definition
of $\mu_{\mathcal{A}^{n}(x)}$,
\begin{eqnarray*}
\int f\, d\mu_{\mathcal{A}^{n}(x)} & = & \frac{1}{\mu(\mathcal{A}^{n}(x))}\int_{\mathcal{A}^{n}(x)}T^{n}f\, d\mu\\
 & = & \mathbb{E}_\mu(T^{n}f\,|\,\mathcal{A}^{n})(x)
\end{eqnarray*}
Writing $g_{n}=\mathbb{E}_{\mu}(T^{n}f\,|\,\mathcal{A}^{n})-T^{n}f$,
it suffices to show that $\lim\frac{1}{N}\sum_{n=0}^{N-1}g_{n}=0$
$\mu$-a.e., and for this it clearly suffices to prove that $\lim\frac{1}{N}\sum_{n=0}^{N-1}g_{kn+p}=0$
$\mu$-a.e. for $0\leq p\leq k-1$.

Now, $g_{n}$ is $\mathcal{A}^{n+k}$-measurable (because $T^{n}f$
is $\mathcal{A}^{n+k}$-measurable); and on the other hand
\[
\mathbb{E}_{\mu}(g_{n}\,|\,\mathcal{A}^{n})=\left(\mathbb{E}_{\mu}(T^{n}f\,|\,\mathcal{A}^{n})-\mathbb{E}_{\mu}(T^{n}f\,|\,\mathcal{A}^{n})\right)=0
\]
Therefore, $\{g_{p+kn}\}_{n=0}^{\infty}$ is an orthogonal
system in $L^{2}(\mu)$, since if $j>i$ then
\begin{eqnarray*}
\int g_{p+ki}\,g_{p+kj}\, d\mu & = & \int\mathbb{E}_{\mu}(g_{p+ki}\,g_{p+kj}|\mathcal{A}^{p+k(i+1)})\, d\mu\\
 & = & \int g_{p+ki}\cdot\mathbb{E}_{\mu}(g_{p+kj}|\mathcal{A}^{p+k(i+1)})\, d\mu\\
 & = & \int g_{p+ki}\cdot0\, d\mu\\
 & = & 0.
\end{eqnarray*}
Since the sequence $\{ g_{p+kn}\}_{n=0}^\infty$ is also uniformly bounded in $L^2(\mu)$, we conclude that $\frac{1}{N}\sum_{n=0}^{N-1}g_{p+kn}\rightarrow 0$
a.e., see for instance  \cite{Lyons88}. (Alternatively, $\{g_{p+kn}\}_{n=1}^{\infty}$ form a sequence
of bounded martingale differences for the filtration $\{\mathcal{A}^{p+kn}\}$,
hence their averages converges a.e. to $0$, see \cite[Chapter 9, Theorem 3]{Feller71}.)

We turn to the second statement. Assume that $\mathcal{A}$ is a topological generator. We will show that the assumption \eqref{eq:boundary-condition} implies that $\nu(\partial A)=0$ for $A\in \mathcal{A}^n, n\in\N$. Fix $n$ and $C=C_n$ as in the statement. For  $\varepsilon>0$ let $f_\varepsilon\in \mathcal{F}$ be such that $1_C\leq f_\varepsilon\leq 1_{C^{(\varepsilon)}}$. Then, using \eqref{eq:martingale-difference} and the hypothesis \eqref{eq:boundary-condition}, we get
\begin{eqnarray*}
\limsup_{k\to\infty} \frac{1}{N_k}\sum_{n=0}^{N_k-1}f_\varepsilon(T^n x) & = & \limsup_{k\rightarrow\infty}\frac{1}{N_{k}}\sum_{n=0}^{N_{k}-1}\int f_\e\,d\mu_{\mathcal{A}^{n}(x)}\\
         &   \le  & \limsup_{k\rightarrow\infty}\frac{1}{N_{k}}\sum_{n=0}^{N_{k}-1} \mu_{\mathcal{A}^{n}(x)}(C^{(\varepsilon)}) \\
         &   =  & o(1)\qquad\mbox{as }\varepsilon\to 0.
\end{eqnarray*}
Since $\mathcal{A}$ is a topological generator, $\mathcal{F}$ is uniformly dense in $C(X)$, so the above conclusion holds also for $f\in C(X)$ satisfying $1_C\le f\le 1_{C^{(\e)}}$. Since $x$ equidistributes for $\nu$ along $\{N_k\}$, this implies that $\nu(C)=0$.
\end{proof}

In the case that $T$ is invertible we consider instead the algebra
$\mathcal{F}^\pm$ of $\mathbb{Q}$-linear combinations of indicators
of sets from $\mathcal{A}^{\pm n}=\bigvee_{i=-n}^{n}T^{i}\mathcal{A}$.
The rest of the proof proceeds as before using the filtration $\mathcal{A}^{\pm n}$.

\section{\label{sec:prlim-on-dim}Preliminaries on dimension}

In this section we summarize some standard and some less well known facts about dimension.

\subsection{\label{sub:dimension}Dimension of measures}

The \emdef{(lower) Hausdorff dimension} of a finite non-zero Borel measure $\theta$ on some metric space is defined by
\[
  \dim\theta=\inf\{\dim A\,:\,\theta(A)>0\;,\;A\mbox{ is Borel}\}.
\]
Here $\dim A$ is the Hausdorff dimension of $A$. We note that this is only one of many possible concepts of dimension of a measure, but it turns out to be the appropriate one for our purposes because of the way it behaves under convolutions, i.e. the resonance and dissonance phenomena discussed in the following sections.

An alternative characterization that we will have occasion to use is given in terms of local dimensions:
\begin{equation}  \label{eq:dim-using-local-dim}
\dim\theta=\essinf_{x\sim\theta}\underline{\dim}(\theta,x),
\end{equation}
where
\[
\underline{\dim}(\theta,x)= \liminf_{r\downarrow 0} \frac{\log \theta (B(x,r))}{\log r}
\]
is the lower local dimension of $\theta$ at $x$. The equivalence is a version of the mass distribution principle, see  \cite[Proposition 4.9]{Falconer03}. Note that this characterization shows that (when the underlying space is compact) the dimension is a Borel function of the measure in the weak$^*$ topology.

We briefly recall some other properties of the dimension which will be used throughout the paper without further reference. Clearly $\dim(\theta|_E)\geq\dim\theta$ for any set $E$ of positive measure, and $\dim$ is invariant under bi-Lipschitz maps (since this is true for the dimension of sets); in particular it is invariant under diffeomorphisms. Dimension also satisfies the relations
\begin{align*}
\dim\sum_i \theta_i &=\inf_i\dim\theta_i,\\
\dim\int \theta_\omega \,dQ(\omega)&\geq \essinf_{\omega\sim Q}\dim \theta_\omega.
\end{align*}
In particular, $\dim\mu=\dim T\mu$ for any map $T$ between intervals that is a piecewise diffeomorphism (such as the maps $T_\beta$ or the Gauss map), as can be seen by writing the measure as a sum over countably many domains where the map is bi-Lipschitz. The same argument shows that dimension is invariant under the quotient map $\mathbb{R}\to\mathbb{R}/\mathbb{Z}$.

Finally, we note that \eqref{eq:dim-using-local-dim} implies that $\dim\mu\times\nu\geq \dim\mu+\dim\nu$ (strict inequality is possible).

\subsection{\label{sub:proj-theorems}Projection theorems}\label{sub:marstrand}

It is a general principle that if $\mu$ is a measure on some space $X$ and $f:X\to Y$ is a ``typical'' Lipschitz map, then the image measure $f\mu$ will have dimension that is ``as large'' as possible: namely, it will have the same dimension as $\mu$ itself if $Y$ is large enough to accommodate this, and otherwise it will be as large as a subset of $Y$ can possibly be, that is, it will have the same dimension as $Y$. Thus one expects $\dim f\mu=\min\{\dim\mu,\dim Y\}$. There are many precise versions of this fact. The most classical is Marstrand's projection theorem, concerning  linear images of sets and measures on $\mathbb{R}^2$. The following version is due to Hunt and Kaloshin \cite[Theorem 4.1]{HuntKaloshin97}.

\begin{thm}\label{thm:classical-marstrand}
If $\eta$ is a probability measure on $\mathbb{R}^2$, then for a.e. $\alpha\in [0,\pi)$, $\dim\pi_\alpha\eta=\min\{1,\dim\eta\}$, where $\pi_\alpha$ is the orthogonal projection onto a line making angle $\alpha$ with the $x$-axis.
\end{thm}

In our applications, $\theta$ will be a product $\mu\times\nu$. In this particular case, we obtain
\begin{cor} \label{cor:marstrand}
Let $\mu,\nu\in\mathcal{P}(\R)$. Then for almost all $t\in\R$,
\[
\dim(\mu*S_t\nu) \ge \min(1,\dim\mu+\dim\nu).
\]
\end{cor}
\begin{proof}
The family of linear maps $\{ P_t(x,y)=x+ty\}$ is a smooth reparametrization of the orthogonal projections $\{ \pi_\alpha\}$, up to affine changes of coordinates which do not affect dimension. Hence, by Theorem \ref{thm:classical-marstrand},
\[
\dim P_t(\mu\times\nu) = \min(1,\dim(\mu\times\nu)) \ge \min(1,\dim\mu+\dim\nu)\quad\text{for a.e. } t.
\]
The corollary follows since $\mu*S_t\nu$ is a restriction of $P_t(\mu\times\nu)$ to a set of positive measure, and restriction does not decrease dimension.
\end{proof}

We will have occasion to use the following refinement of the above.
\begin{thm} \label{thm:marstrand-dimofexceptions}
If $\mu,\nu$ are Borel probability measures on $\mathbb{R}$ such that $\dim\mu+\dim\nu>1$, then
\[
\dim\{t\in\mathbb R: \dim(\mu*S_t\nu) < 1\} < 1.
\]
\end{thm}
\begin{proof}
Falconer \cite{Falconer82} essentially proved the corresponding result for Hausdorff dimensions of sets, we indicate how to modify his proof to work with dimension of measures (the argument is standard). Let $P_t(x,y)=x+ty$. In the course of the proof of \cite[Theorem 1]{Falconer82} it is shown that if $\eta$ is a Borel probability measure on $\R^2$ such that
\begin{equation} \label{eq:finite-energy}
\int \int \frac{d\eta(x)d\eta(y)}{|x-y|^s} < \infty
\end{equation}
for some $s>1$, then the set $E$ of parameters $t$ such that the projection $P_t\eta$ is not absolutely continuous, satisfies $\dim(E)\le 2-s<1$ (as above, Falconer worked with orthogonal projections, but by reparametrization the same holds for the family $\{P_t\}$).

Let $\rho=\mu\times\nu$. We only need to show that $\dim(E)<1$, where
\[
E=\{\alpha: P_\alpha\rho \text{ is not absolutely continuous}\}.
\]
We have $\dim\rho\ge\dim\mu+\dim\nu>1$. Using Equation \eqref{eq:dim-using-local-dim}, it follows that there is $s_0>1$ such that
\[
\liminf_{r\downarrow 0} \frac{\log \rho(B(x,r))}{\log r} \ge s_0\quad\text{for }\rho\text{-a.e. }x.
\]
By Egorov's Theorem, for any $\e>0$ there are a set $A_\e$ with $\rho(A_\e)>1-\e$ and a constant $r_\e>0$ such that
\[
\rho(B(x,r)) \le r^{(1+s_0)/2}\quad\text{for all }x\in A_\e, 0<r<r_\e.
\]
It follows that $\eta:=\rho|_{A_\e}$ satisfies \eqref{eq:finite-energy} with $s=1+(s_0-1)/4$ (say). Hence $\dim(E_\e)\le 2-s$, where $E_\e=\{\alpha: P_\alpha\rho|_{A_\e} \text{ is singular}\}$. Since $E\subseteq \bigcup_{n\in \N} E_{1/n}$ the result follows.

\end{proof}

\subsection{Further facts on dimension}\label{sub:further-dimension}

For the part of the proof of Theorem \ref{thm:application-HostLindenstrauss} dealing with invariance under $C^2$ diffeomorphisms, we will need some classical but perhaps less well-known facts about dimension. The material below will not be used anywhere except in this application.

It is always true that $\dim(A\times B)\ge \dim(A)\times \dim(B)$ for Borel sets $A,B$; however, the inequality may be strict. More generally, there is a ``Cavalieri inequality'' for Hausdorff dimensions. To get inequalities in the opposite direction, one needs to consider also packing dimension $\dim_P$. The interested reader may consult e.g. \cite{Mattila95} for its definition, but we shall only require the property given in the following proposition.

\begin{prop} \label{pro:properties-dim-sets}
Let $E\subseteq\R^{d_1+d_2}$ be a Borel set.
\begin{enumerate}
\item Suppose there is a set $A\subseteq\R^{d_1}$ of positive Lebesgue measure, such that for $x_0\in A$, the fiber $\{y:(x_0,y)\in E\}$ has Hausdorff dimension at least $\alpha$. Then $\dim(E)\ge d_1+\alpha$.
\item Let $P_i$ be the coordinate projection onto $\R^{d_i}$. Then $\dim(E)\le \dim(P_1 E)+\dim_P(P_2 E)$.
\end{enumerate}
\end{prop}

The first part follows from \cite[Theorem 7.7]{Mattila95}, and the second from \cite[Theorem 8.10]{Mattila95}

We now turn to measures. In a similar way to our definition of lower Hausdorff dimension $\dim$, we may define \emdef{upper packing dimension} $\dim_P$ as
\[
\dim_P(\mu) =\inf\{ \dim_P(E):\mu(E)=1 \}.
\]
(Note that in the definition of $\dim$ the infimum is taken over sets of positive measure; here, it is taken over sets of full measure.) The following is an analog of Proposition \ref{pro:properties-dim-sets} for measures.

\begin{lem} \label{lem:properties-dim-measures}
Let $\mu$ be a measure on $\R^{d_1+d_2}$, and let $P_i$ be the coordinate projection onto $\R^{d_i}$.
\begin{enumerate}
\item \label{it:cavalieri-ineq} Suppose $P_1\mu$ is absolutely continuous, and $\dim(\mu_{x_0})\ge \alpha$ for $P_1\mu$-a.e. $x_0$, where $\mu_{x_0}$ is the conditional measure on the fiber $\{(x,y):x=x_0\}$. Then $\dim\mu\ge d_1+\alpha$.
\item \label{it:upper-bound-from-projections} $\dim\mu\le \dim(P_1 \mu)+\dim_P(P_2 \mu)$.
\end{enumerate}\end{lem}
\begin{proof}
For the first part, suppose $\mu(E)>0$. Then there is a set $A$ with $P_1\mu(A)>0$ (and hence $A$ has positive Lebesgue measure) such that $\mu_x(E)>0$ for almost all $x\in A$. The claim then follows from the corresponding statement for sets. The second part is established in a similar manner.
\end{proof}

Finally, recall that a measure $\mu$ is \emdef{exact dimensional} if the local dimension
\[
\lim_{r\downarrow 0} \frac{\log\mu(B(x,r))}{\log r}
\]
exists and is $\mu$-a.e. constant. For exact dimensional measures $\mu$, it is well known that $\dim\mu=\dim_P\mu$, with both dimensions agreeing with the almost sure value of the local dimension; see \cite[Proposition 2.3]{Falconer97}. If $\beta>1$ is Pisot and $\mu$ is $T_\beta$-ergodic, then $\mu$ is exact dimensional. This well known fact follows from the Shannon-McMillan-Breiman Theorem and, in the Pisot case, a classical Lemma of Garsia (see Lemma \ref{lem:Garsia} below).

\section{\label{sec:efds-spectra-and-phase}Ergodic fractal distributions, spectra and phase}

\subsection{\label{sub:ergodcity-and-spectrum}Ergodicity and spectrum}

Below we prove some basic facts relating the spectrum of a flow to the equidistribution properties of points under individual maps in the flow. The discussion is mostly valid for general flows on metric spaces but for simplicity we formulate them for $(\mathcal{M}^\sqr,S)$.

\begin{prop}\label{pro:ergodicity-and-spectrum}
If $P\in\mathcal{D}$ is $S$-ergodic and $t_0>0$, then $P$ is $S_{t_0}$-ergodic if and only if no non-zero multiple of $1/t_0$ is in the pure point spectrum of $P$.
\end{prop}

\begin{proof}
$S$ acts on the ergodic decomposition with respect to $S_{t_0}$: $P=\int P_\mu\,dP(\mu)$. Clearly this action is $t_0$-periodic. Thus the factor of $(P,S)$ with respect to the $\sigma$-algebra $\mathcal{E}$ of invariant sets factors through the standard translation action of $\mathbb{R}$ on $\mathbb{R}/t_0\mathbb{Z}$. The only factors of the latter action are the trivial one, in which case  $\mathcal{E}$ is trivial and $P$ is $S_{t_0}$-ergodic, or an action isomorphic to the translation action of $\mathbb{R}$ on $\mathbb{R}/(t_0/k)\mathbb{Z}$ for some $k\in\mathbb{Z}\setminus\{0\}$, in which case this factor map defines an eigenfunction with eigenvalue $k/t_0$.
\end{proof}

\begin{lem}\label{lem:discrete-equidistribuion-a.e.}
Let $P$ be $S$-ergodic and $t_0>0$. Then  $P$-a.e. $\mu$ equidistributes under $S_{t_0}$ for an $S_{t_0}$-ergodic distribution $P_\mu$, and $P=\int P_\mu\,dP(\mu)$ is the ergodic decomposition of $P$ under $S_{t_0}$. If no multiple of $1/t_0$ is in $\Sigma(P,S)$, then $P_\mu=P$ a.s.
\end{lem}
\begin{proof}
Let $P=\int P_\mu\,dP(\mu)$ be the ergodic decomposition of $P$ with respect to the measure-preserving map $S_{t_0}$. By the ergodic theorem, for $P$-a.e. $\mu$,  $P_\mu$-a.e. $\nu$ equidistributes for $P_\nu$. the first statement follows. For the second statement, if $k/t_0\notin\Sigma(P,S)$ for all non-zero integers $k$, then by the previous Proposition $P$ is $S_{t_0}$-ergodic, and so $P_\mu=P$ a.s.
\end{proof}

We turn to distributions generated by a measure $\mu\in\mathcal{P}(\mathbb{R})$. Given $t_0>0$, we say that a distribution $P$ is $t_0$-generated by $\mu$ at $x$ if $\mu^x$ equidistributes for $P$ under the discrete semigroup $\{S_{kt_0}\}_{k\in\mathbb{N}}$, that is, the sequence $\{\mu_{x,kt_0}\}_{k=0}^\infty$ equidistributes for $P$.

We have seen that if $k/t_0\not\in\Sigma(P,S)$ for all non-zero integers $k$, then $P$-a.e. $\mu$ $t_0$-equidistributes for $P$. The next result says that the same is true for any measure $\mu$ that generates $P$.

\begin{lem} \label{lem:discrete-equidistribution-and-spectrum}
Suppose $\mu$ generates an $S$-ergodic distribution $P$ and no non-zero integer multiple of $t_0$ is an eigenvalue of $(P,S)$. Then $P$ is $t_0$-generated at $\mu$-a.e. $x$.
\end{lem}

This is, essentially, the following well-known fact from ergodic theory, whose proof we provide for completeness:

\begin{lem}
\label{lem:discrete-equidistribution-and-spectrum-continuous}Let $W=(W_{t})_{t>0}$
be a continuous flow on a compact metric space $X$. Suppose $\theta$
is a $W$-invariant and ergodic measure which does not have $k/t_{0}$ is its
pure point spectrum for any $k\in\mathbb{Z}\setminus\{0\}$. Then
any point $x$ which equidistributes for $\theta$ under $W$ equidistributes
for $\theta$ also under the ``time $t_{0}$'' map $W_{t_{0}}$.\end{lem}

\begin{proof}
As in Lemma \ref{lem:discrete-equidistribuion-a.e.}, the spectral hypothesis implies ergodicity of $\theta$ under the map $W_{t_{0}}$.
Now suppose that $x$ equidistributes for a measure $\theta'$ under
$W_{t_{0}}$ along a sequence $N_{k}\rightarrow\infty$; it suffices
to prove $\theta'=\theta$. By continuity, $\theta'$ is $W_{t_{0}}$-invariant, so $W_t\theta'$ is $W_{t_0}$-invariant for every $t$. Let $\rho=\frac{1}{t_{0}}\int_{0}^{t_{0}}W_{t}\theta'\, dt$.
Then for every $f\in C(X)$,  \begin{eqnarray*}
  \int f\, d\rho & = & \lim_{k\rightarrow\infty}\frac{1}{t_{0}}\int_{0}^{t_0}\frac{1}{N_{k}}\sum_{n=0}^{N_{k}-1}f(W_{t_{0}}^{n} W_t x) dt\\
                 & = & \lim_{k\rightarrow\infty}\frac{1}{N_{k}t_{0}}\int_{0}^{N_{k}t_{0}}f(W_{t}x)\, dt\\
                 & = & \int f\, d\theta,
\end{eqnarray*}
where the last equality is because $x$ equidistributes for $\theta$. Thus $\rho=\theta$, i.e.
$\frac{1}{t_{0}}\int_{0}^{t_{0}}W_{t}\theta'dt=\theta$. Since $\theta$
is $W_{t_{0}}$-ergodic and this is a representation of $\theta$
as the integral of $W_{t_{0}}$-invariant measures, we conclude that
$W_{t}\theta'=\theta$ for a.e. $t$. Since $\theta$ is $W$-invariant
this holds for $t=0$, i.e. $\theta'=\theta$, as desired.
\end{proof}

A priori this does not apply in our situation, because the topological assumptions are not satisfied ($S$ acts discontinuously, and is not everywhere defined on $\mathcal{P}(\mathcal{P}([-1,1]))$). However, the only place in the proof that continuity was used was in the assertion that $\theta'$ is $W_{t_0}$-invariant. In the context of Lemma \ref{lem:discrete-equidistribution-and-spectrum} this is true at $\mu$-a.e. point by \cite[Theorem 1.7]{Hochman12b}. Thus, we have proved Lemma \ref{lem:discrete-equidistribution-and-spectrum}.

\subsection{\label{sub:EFDs}Ergodic fractal distributions}

\begin{defn} \label{def:quasi-palm}
An $S$-invariant distribution $P\in\mathcal{D}$ is \emdef{$S$-quasi-Palm} if for every Borel set $B\subseteq\mathcal{M}^\sqr$, $P(B)=1$ if and only if for every $t>0$, $P$-almost every measure $\eta$ satisfies $\eta_{x,t}\in B$ for $\eta$-almost all $x$ such that $[x-e^{-t},x+e^{-t}]\subseteq [-1,1]$.
\end{defn}

\begin{defn}\label{def:efd}
A distribution $P\in\mathcal{D}$ which is supported on $\mathcal{M}^\square$, $S$-invariant and satisfies the $S$-quasi-Palm property is called a \emdef{fractal distribution}, or \emdef{FD}. If, in addition, $P$ is $S$-ergodic, then $P$ is called an \emdef{ergodic fractal distribution}, or \emdef{EFD}.
\end{defn}

This definition differs slightly from the one introduced and studied in \cite{Hochman12b}. More precisely, the notion of quasi-Palm in \cite{Hochman12b} is suited for distributions on Radon measures on $\mathbb{R}$, rather than distributions on probability measures on $[-1,1]$, and the notion of EFDs there is for distributions on Radon measures that are invariant under the action of a semigroup $S^*$, which is defined similarly to $S$ but without restricting the measures to a bounded interval, so that $S^*$ acts on measures of unbounded support (our $S$ is denoted by $S^\Box$ in \cite{Hochman12b}). For this reason, in the definition of quasi-palm measure given in \cite{Hochman12b} there is no need to assume that $[x-e^{-t},x+e^{-t}]\subseteq [-1,1]$, and it has $\mu^x$ in place of $\mu_{x,t}$. However, it is proved in \cite[Lemma 3.1]{Hochman12b} that $S$-invariant and $S^*$-invariant distributions are canonically in one-to-one correspondence. Hence any EFD according to our definition arises as the push-forward of an EFD in the sense of \cite{Hochman12b} under the map $\mu\mapsto \mu|_{[-1,1]}$. Therefore all results proved for EFDs in \cite{Hochman12b} continue to be valid with our definition of EFD. In particular, the following is proved in \cite[Theorem 1.7]{Hochman12b}.

\begin{thm} \label{thm:generated-dists-are-fds}
For $\mu$ almost all $x$, any distribution $P$ generated by $\mu$ at $x$ along a sequence of times $T_i$ is a FD (i.e. it is $S$-invariant and automatically satisfies the $S$-quasi-Palm property).

In particular, if $\mu$ generates an $S$-ergodic distribution $P$, then $P$ is an EFD.
\end{thm}

For the rest of the section we fix an EFD $P$, and shall draw some simple but important conclusions about it. We will repeatedly use the following consequence of the $S$-quasi-palm property:

\begin{lem}\label{lem:consequence-quasi-palm}
Let $P$ be an EFD, and $B\subseteq\mathcal{M}^\sqr$ a Borel set with the property that $\eta\in B$ whenever $S_t\eta\in B$ for some $t$. Then $P(B)=1$ if and only if for $P$-almost all $\eta$ and $\eta$-almost all $x$, the translation $\eta^x$ is in $B$.
\end{lem}

As a first application, we have:
\begin{lem}\label{lem:fd-measures-generate-discretely}
Let $P$ be an EFD. Fix $t_0>0$. $P$-a.e. $\mu$ generates $P$, and $t_0$-generates an $S_{t_0}$-ergodic component of $P$ at  $\mu$-a.e. point $x$.
\end{lem}
\begin{proof}
Let $B$ be the set of $\mu$ such that $\mu$ generates $P$ and $t_0$-generates an $S_{t_0}$-ergodic component of $P$ at $0$; then $P(B)=1$ by ergodicity. Also, $\eta \in B$ whenever $S_t\eta \in B$, so the lemma follows from Lemma \ref{lem:consequence-quasi-palm}.
\end{proof}

Recall that an $S$-invariant distribution is trivial if it is supported on the $S$-fixed point $\delta_0$.

\begin{lem}\label{lem:no-atoms}
If $P$ is a non-trivial EFD then $P$-almost all measures are non-atomic.
\end{lem}
\begin{proof}
Let $a(\mu)=\mu(\{0\})$. It is clear that $a(S_t\mu)\geq a(\mu)$, by definition of $S$, so $a$ is a.s. constant. Also it is clear that if $a(\mu)>0$ then $a(S_t\mu)\to 1$ as $t\to\infty$, so if that were the case, $a=1$ $P$-a.s. But this would imply that $\mu=\delta_0$ a.s. and so $P$ is trivial, contrary to assumption.  Hence $a=0$ $P$-a.s.; using Lemma \ref{lem:consequence-quasi-palm} applied to the set $\{\nu\,:\,a(\nu)=0\}$, we find that $P$-a.e. $\nu$ satisfies $a(\nu^x)=0$ for $\nu$-a.e. $x$, so $\nu$ is non-atomic.
\end{proof}

\begin{lem}\label{lem:continuity-of-EFD-measures}
Suppose that $\mu$ $t_0$-generates $P$ and $P$ is supported on non-atomic measures. For every $\varepsilon>0$ there is a $\rho>0$ such that \[
  \limsup_{N\to\infty} \frac{1}{N}\sum_{n=0}^{N-1}\sup \mu_{x,t_0 n}(I)<\varepsilon\quad\text{for }\mu\text{-almost all } x.
\]
where the supremum is over intervals $I\subseteq [-1,1]$ of length $|I|<\rho$.
\end{lem}
\begin{proof}
Fix $\e>0$ and let $\mathcal{C}_\rho$ denote the set of measures $\eta$ such that $\eta(I)<\varepsilon$ for every open interval of length $|I|<\rho$. Note that $\mathcal{C}_\rho$ is open.

By the fact that $P$ gives no mass to measures with atoms, for $P$-a.e. $\eta$ there is a $\rho=\rho_\eta>0$ depending on $\eta$ such that $\sup \eta(I)<\varepsilon$ where $I$ ranges over open intervals of length $\rho_\eta$. It follows that there is a $\rho$ such that with $P$-probability $>1-\varepsilon$ we have $\rho<\rho_\eta$, and in particular $P(\mathcal{C}_\rho)>1-\varepsilon$. Since $\mu$ $t_0$-generates $P$, we find that
\[
   \limsup_{N\to\infty} \frac{1}{N}\sum_{n=0}^{N-1}\delta_{\mu_{x,t_0 n}\in\mathcal{C}_\rho}\geq P(\mathcal{C}_\rho)>1-\varepsilon
\]
as required.
\end{proof}

It is not hard to show that the same conclusion holds if one assumes only that $\mu$ generates a non-trivial $P$ (without necessarily $t_0$-generating it), but we will not use this fact.

In fact, not only are $P$-typical measures non-atomic; they also have positive dimension:

\begin{prop}\label{pro:positive-dim}
Let $P$ be an EFD. There is a number $\delta$ such that $P$-a.e. $\nu$ has $\dim\nu=\delta$. If $P$ is nontrivial then $\delta>0$.
\end{prop}
\begin{proof}
This follows from \cite[Lemma 1.18]{Hochman12b}; we include a proof for completeness. For the first statement, restriction can only increase dimension, and scaling does not affect it, so for any measure $\nu$ we have $\dim S_t\nu\geq\dim \nu$. By ergodicity, the dimension is $P$-a.s. equal to some constant $\delta\ge 0$.

Now assume that $P$ is nontrivial, we need to show that $\delta>0$. We will use the characterization of dimension using local dimension, recall Equation \eqref{eq:dim-using-local-dim}. Write
\[
f(\nu) = \liminf_{r\downarrow 0} \frac{\log\nu([-r,r])}{\log r}.
\]
By Lemma \ref{lem:consequence-quasi-palm}, it is enough to verify that there is $\delta>0$ such that $f(\nu)\ge\delta$ for $P$-a.e. $\nu$ (note that the set $B=\{ \nu: f(\nu)\ge\delta\}$ satisfies $S^t\nu\in B\Rightarrow \nu\in B$). But $f$ is $S$-invariant, whence by ergodicity we only need to check that $f(\nu)>0$ on a set of positive $P$-measure.

Now Lemma \ref{lem:no-atoms} and $S$-invariance ensure that $g(\nu)=-\log\nu([-1/2,1/2])$ satisfies $\int g\,dP>0$. By the ergodic theorem applied to the (possibly non-ergodic) discrete-time system $S_{\log 2}$,
\[
\lim_{N\to\infty}\frac{\log\nu([-2^{-N},2^{-N}])}{N\log 2} = \lim_{N\to\infty} \frac{1}{N\log 2} \sum_{n=0}^{N-1} g(S^{n\log 2}\nu)
\]
converges almost everywhere to a function of $\nu$ with strictly positive integral; but the left-hand side equals $f(\nu)$, so this completes the proof.
\end{proof}

We will also need to know that $P$-typical measures are not ``one-sided at small scales''.

\begin{prop} \label{pro:large-mass-near-0}
Let $P$ be an EFD. For every $\rho>0$, for $P$-a.e. $\nu$ we have $\inf\nu(I)>0$,
where $I\subseteq[-1,1]$ ranges over closed intervals of length $\rho$
containing $0$.
\end{prop}
\begin{proof}
Let $B=\{\nu:\nu[-\e,0]=0\text{ for some }\e>0\}$. It is enough to show that $P(B)=0$. Indeed, if this is true then by symmetry also $P(B')=0$ where $B'=\{\nu:\nu([0,\e])=0\text{ for some }\e>0\}$, and the claim follows since any interval of length $\rho$ containing $0$ contains either $[-\rho/2,0]$ or $[0,\rho/2]$.

Since $B$ is $S$-invariant, by ergodicity we only need to show that $P(B)<1$. Suppose otherwise. Since $S^t\mu\in B$ implies that $\mu\in B$, it  follows from Lemma \ref{lem:consequence-quasi-palm} that, for $P$-typical $\nu$ and $\nu$-typical $x$, there is $\e(x)$ such that $\nu([x-\e(x),x])=0$. Take $\e>0$ such that $\nu(A)>0$, where $A=\{x:\e(x)\ge \e\}$. The restriction $\nu|_A$ has the property that the distance between any two distinct points in its support is at least $\e$. However this can only happen for discrete measures, and we have already established in Lemma \ref{lem:no-atoms} that $P$-typical measures have no atoms. Hence $P(B)<1$ and therefore $P(B)=0$, as claimed.

\end{proof}

\subsection{\label{sub:phase}Phase and synchronization}

Suppose that $\mu$ generates $P$ and $t_0$-generates an $S_{t_0}$-ergodic distribution $P_x$ at $\mu$-typical points $x$. Let $\varphi$ be an eigenfunction of the flow $(P,S)$ for some eigenvalue $k/t_0$. Since $\varphi$ is $S_{t_0}$-invariant, it is almost surely constant on  each ergodic component of $P$ under $S_{t_0}$, hence it is $P_x$-a.s. constant for $\mu$-a.e. $x$. This allows us to define the \emdef{phase} of $\mu$ at  $x$ to be the a.s. value of $\varphi$ on $P_x$. We denote the phase by $\varphi_\mu(x)$, and claim that it is a measurable function of $x$. Indeed, write $\varphi$ as the increasing limit of simple functions $\varphi_n$, and note that $\varphi_\mu(x) = \int \varphi\, dP_x =  \lim_{n\to\infty} \int \varphi_n\, dP_x$. The map $x\mapsto P_x$ is measurable in $x$, since $P_x$ arises as an almost-sure limit of measurable functions of $x$, and hence $x\mapsto\int\phi_n\,dP_x$ is measurable for each $n$. By the limit above, also $\phi_\mu$ is.

The push-forward of $\mu$ to the unit circle by $x\mapsto\varphi_\mu(x)$ gives a measure $\theta=\theta_\mu$ which  describes the distribution of phases, and is called the \emdef{phase measure}.\footnote{This definition of the phase and phase measure differs from that in \cite[Definition 2.6]{Hochman12}, but the two definitions coincide for measures for which both definitions apply. One might say that the definition given here is absolute (given $\varphi$), while the definition in \cite{Hochman12} is relative, as it measures difference in phase between sceneries at pairs of $\mu$-typical points.}

\begin{lem}\label{lem:P-ae-phase}
For $P$-typical $\nu$, let $P_\nu$ denote the $S_{t_0}$-ergodic component of $P$ to which $\nu$ belongs. Then for $P$-a.e. $\nu$, the phase of $\nu$ is well defined at $0$ and is equal to $\varphi(\nu)$.
\end{lem}
\begin{proof}
Fix  an $S_{t_0}$-ergodic component $P'$ of $P$. Let $z$ denote the $P'$-a.s. value of $\varphi$. Now, for $P'$-a.e. $\nu$ we know that $\varphi(\nu)=z$ and, by the ergodic theorem, that $\nu$ equidistributes for $P'$ under $S_{t_0}$. This shows that the phase of $\nu$ is well defined at $0$ and equal to $z$. Since $P$ is the integral of its ergodic components, the claim follows.
\end{proof}

\begin{prop}\label{pro:synchronization}
For $P$-a.e. $\nu$, the function $\varphi_\nu$ is $\nu$-a.e. constant and $\theta_\nu=\delta_{\varphi(\nu)}$.
\end{prop}
\begin{proof}
By the $S$-quasi-Palm property and the last lemma, it is clear that for $P$-a.e. $\nu$ and $\nu$-a.e. $x$, the eigenfunction $\varphi$ is well-defined on $S_t(\nu^x)$ for all large enough $t$, and that this value is the phase of the distribution that is $t_0$-generated by $\nu^x$. Since  $x\mapsto \varphi_\nu(x)$ is measurable,  it is enough to show that $P$-almost all $\nu$ and all $\e>0$,
\[
\int\int|\varphi_\nu(y')-\varphi_\nu(y'')|\,d\nu(y')\,d\nu(y'') < \e
\]
Write $A_\e$ for the set of $\nu$ for which the above holds; we aim to show $P(A_\e)=0$. Let
\[
B_\e = \{ \nu: S_t\nu\in A_\e \text{ for sufficiently large } t\}.
\]
By invariance, it is enough to show that $P(B_\e)=1$.

By the Besicovitch differentiation theorem \cite[Corollary 2.14(2)]{Mattila95}, for $\nu$-almost all $x$,
\[
\lim_{t\to\infty} \frac{\int_{[x-e^{-t},x+e^{-t}]} |\varphi_\nu(x)-\varphi_\nu(y)|d\nu(y)}{\nu([x-e^{-t},x+e^{-t}])} \to 0\quad\text{as }t\to\infty,
\]
and therefore
\[
\lim_{t\to\infty} \frac{\int_{[x-e^{-t},x+e^{-t}]^2} |\varphi_\nu(y')-\varphi_\nu(y'')|d\nu(y')d\nu(y'')}{\nu([x-e^{-t},x+e^{-t}])^2} \to 0\quad\text{as }t\to\infty.
\]
From the eigenfunction property, $\varphi_{\nu_{x,t}}(y)=e(-\alpha t)\varphi_\nu(x+e^{-t}y)$ for all $t$, $\nu$-a.e. $x$ and $\nu_{x,t}$-a.e. $y$. It follows that for $\nu$-a.e. $x$, the measure $\nu_{x,t}$ is in $A_\e$ for sufficiently large $t$, i.e. $\nu^x\in B_\e$. But then we conclude from Lemma \ref{lem:consequence-quasi-palm}  that $P(B_\e)=1$, as desired.
\end{proof}

Finally we consider the effect of perturbation on the generated distributions and the phase of a measure.

\begin{lem} \label{lem:image-of-USM}
Let $\nu\in\mathcal{P}(\mathbb{R})$.
\begin{enumerate}
\item Let $f\in L^1(\nu)$, $f\ge 0$ and $\int f d\nu>0$, and write $d\nu'=f\,d\nu$. Then for $\nu'$-a.e. $x$, the sceneries of $\nu$ and of $\nu'$ at $x$ are asymptotic. In particular, if $\nu$ generates $P$, then so does $\nu'$.
\item Let $I$ be an interval and $f:I\to J$ an orientation-preserving diffeomorphism. Let $\nu'=f(\nu)$. Then for $\nu$-a.e. $x$, the sceneries  $\nu_{x,t}$ and $\nu'_{f(x),t-\log f'(x)}$ are mean-asymptotic in $\mathcal{P}([-1,1])$ in the sense that \[
  \lim_{T\to\infty}\frac{1}{T} \left( \int_0^T F(\nu_{x,t})\,dt - \int_0^T F(\nu'_{f(x),t-\ln f'(x)})\,dt\right) = 0\;\;\;\;\textrm{for all }F\in C([-1,1])
\]
and similarly when one averages at discrete time steps of some size $t_0$. In particular, if $\nu$ generates $P$ at $x$ then $\nu'$ generates $P$ at $f(x)$.
\end{enumerate}
\end{lem}
\begin{proof}
The first part is an immediate consequence of the Besicovitch differentiation theorem (see Mattila \cite[Corollary 2.14(2)]{Mattila95}, or \cite{Hochman12b} for more detail).

The second part can be proved by adapting the argument in  Proposition 1.9 of \cite{Hochman12b} or the forthcoming paper of Aspenberg, Ekstr\"{o}m, Persson and Shmeling \cite{AspbergEkstromPerssonSchmeling2013}.\footnote{This lemma first appeared as Lemma 2.3 of \cite{Hochman12} but the statement there incorrectly omits the ``almost every'' quantifier over $x$.} Here we only give a sketch. Consider the maps $g_t(y)= e^t\cdot(y-x)$ and $h_t(y)= e^t\cdot (f(y)-f(x))$, so that $\nu_{x,t}=a_t\cdot g_t(\nu)|_{[-1,1]}$ and $\nu'_{x,t}=b_t\cdot h_t(\nu)|_{[-1,1]}$ for normalizing constants $a_t,b_t$ (we suppress the dependence on $x$ in the notation). Using the linear approximation of $f$ at $x$, we see that the uniform distance $\e(t)$ between the maps $g_t$ and $h_{t-\log f'(x)}\circ f$ on $[x-2e^{-t},x+2e^{-t}]$ tends to $0$ as $t\to\infty$. Thus we will be done if we show that for $\nu$-a.e. $x$ we have  $a_t/b_t\to 1$ in the mean (Cesaro) sense. Now, for a given $\delta>0$, in order to have $|a_t/b_t-1|>\delta$, we must have $\left|\frac{\nu(B_{e^{-t-\e(t)}}(x))}{\nu(B_{e^{-t+\e(t)}}(x))}-1\right|>\frac{\delta}{100}$. If this were to happen for a non-negligible proportion of $t$s in arbitrarily long intervals $[0,T_i]$ we would conclude that there is a distribution $P$ generated by $\nu$ at $x$ along the times $T_i$, such that, with positive $P$-probability, a measure $\theta$ satisfies $\theta(\{\pm 1\})>0$. This is impossible by Theorem \ref{thm:generated-dists-are-fds}, Lemma \ref{lem:no-atoms} and the ergodic decomposition.

In the discrete time case, suppose that when averaged at steps of size $t_0$ the two sceneries are not a.s. mean-asymptotic. Passing to a subsequence, the we find that for a positive $\mu$-proportion of $x$, there is a subsequence along which $\mu$ generates some distribution $P_x$ $t_0$-discretely at $x$, and $P_x$ gives positive mass to measures with atoms at $\pm 1$. But then for $\mu$-a.e. such $x$ one sees that  $P'_x=\int_{-t_0}^{0}S_tP_x\,dt$ is a FD supported on measures that have atoms at non-zero points, and we know this is impossible, because each ergodic component of $P$ is an EFD \cite{Hochman12} and is either trivial, in which case its measures have an atom only at $0$, or non-trivial, in which case Lemma \ref{lem:no-atoms} applies (since the space of atomic measures with atoms is not closed, some more care must be taken in the last step, and one needs to use the fact that for $P$ there is already a positive probability of finding atoms of mass bounded away from zero at locations bounded away from $0,\pm 1$, and this translates to $P'$. We omit the details).

\end{proof}

\begin{cor}\label{cor:phase-shift}
If $\mu$ generates $P$ $t_0$-discretely, $P$ is $S_{t_0}$-ergodic and $t_0\in\Sigma(P,S)$ with eigenfunction $\varphi$, then
\begin{enumerate}
  \item If $\nu\ll\mu$, then $\theta_\nu$ is well defined and $\theta_\nu\ll\theta_\mu$.
  \item If $f\in\diff^1(\mathbb{R})$ and $\nu=f(\mu)$, then $\theta_\nu$ is well defined and
  \[
   \theta_\nu = \int \delta_{e(-t_0\log f'(x))\varphi_\mu(x)}\,d\mu(x).
\]
\end{enumerate}
\end{cor}
\begin{proof}
For (1), by the previous lemma, if $\nu\ll\mu$ then for $\nu$-a.e. $y$, the distribution $t_0$-generated by $\mu$ and $\nu$ at $y$ is the same, and the claim follows. For (2), fixing a $\mu$-typical $x$, by the second part of the previous lemma, $\mu_{t,x}$ and $\nu_{f(x),-\log f'(x)}$ generate that same distribution $t_0$-discretely. Hence\footnote{Here we use the fact that although $S_{-\log f'(x)}$ is not continuous, it is continuous on the set of non-atomic measures, and hence on a set of full measure for $P$, since the non-triviality of $\Sigma(P)$ implies that $P$ is non-trivial, hence supported on non-atomic measures.} $\nu$ generates $S_{-\log f'(x)}P$ $t_0$-discretely at $f(x)$, and by the eigenfunction property,
\[
\varphi_\nu(fx)=e(-t_0 \log f'(x)) \varphi_\mu(x),
\]
from which we deduce (2).
\end{proof}

\section{\label{sec:proof-of-WM-case}Proof of theorem \ref{thm:WM-pisot-case}}

\subsection{A sketch of the proof}\label{sub:sketch-of-proof}

We start by explaining the main steps involved in the proof of Theorem \ref{thm:WM-pisot-case}. This strategy will also apply for the generalizations considered in Section \ref{sec:refinements}, with suitable modifications.

We start with a measure $\mu$ on $[0,1]$ generating an EFD $P$ such that $k/\log\beta\notin\Sigma(P,S)$ for $k\in\mathbb{Z}\setminus \{0\}$ for Pisot $\beta>1$. We fix a $\mu$-typical $x$ and suppose that $x$ equidistributes under $T_\beta$ for a measure $\nu$ along some subsequence $N_j$; our job is to show that $\nu$ is in fact the Parry measure $\lambda_\beta$. To accomplish this, there are three main steps involved:

\begin{enumerate}
\item The first step is to use Theorem \ref{thm:local-average} and the spectral hypothesis to establish that $\nu$ can be represented as a superposition of measures drawn according to $P$, each of them suitably translated, restricted and normalized. See Theorem \ref{thm:integral-representation} and the ensuing discussion for the general Pisot case.
\item We show that any $T_\beta$-invariant measure of positive dimension, other than the Parry measure, resonates with measures of arbitrarily large dimension (see Section \ref{sub:resonance-and-dissonance} for the definition of resonance and dissonance). This is stated in Theorem \ref{thm:existence-of-resonant-measures} and proved in Section \ref{sec:resonance}.
\item Using the first step, the $S$-invariance of $P$ and Marstrand's Theorem, we show that $\nu$ dissonates with arbitrary measures of sufficiently large dimension (this step uses the nontriviality of $P$). Hence, in light of the second step, $\nu$ must be the Parry measure. This step is carried over in Section \ref{sub:proof-of-WM-case}.
\end{enumerate}

We note that both the first and second steps use the algebraic assumption on $\beta$ (in each case it can be slightly relaxed, but in different directions).

\subsection{An integral representation}\label{sub:integral-representation}

We begin with the details. From now on, we specialize to the interval $[0,1]$ and to maps of the form $T_n:x\mapsto n x\bmod 1$ for an integer $n\geq 2$. We comment on the Pisot case afterwards. Let $\mu\in\mathcal{P}([0,1])$ be a measure that generates a distribution $P$ satisfying the spectral hypothesis in Theorem \ref{thm:WM-case} (we do not assume that $\mu$ is $T_\beta$-invariant; in fact, we will eventually apply the result of this section to measures $\mu$ which are invariant under a \textit{different} dynamics). We shall obtain a certain integral representation of the measures for which $\mu$-typical points equidistribute along sub-sequences.

Let $\mathcal{A}$ denote the partition of $[0,1]$ into $n$-adic intervals, $[j/n,(j+1)/n)$. Note that $\delta_y*\nu$ is the translate of the measure $\nu$ by $y$. Fixing $x$, we claim that
\begin{equation}
  \mu_{\mathcal{A}^k(x)} = c_k\cdot (\delta_{y_k} *\mu_{x,k\log n})|_{[0,1]} \label{eq:A-adic-vs-scenery}
\end{equation}
for some normalizing constant $c_k$ and a number $y_k\in[0,1]$. Indeed, $\mu_{x,k\log n}$ is the restriction of $\mu$ to the interval $I$ of side $2\cdot n^{-k}$ centered at $x$, re-scaled to $[-1,1]$ and normalized; while $\mu_{\mathcal{A}^k(x)}$ is obtained similarly from the restriction of $\mu$ to an interval $J=\mathcal{A}^k(x)$ of length $n^{-k}$ around $x$, re-scaled to the interval $[0,1]$ and normalized. Since $J\subseteq I$, the representation \eqref{eq:A-adic-vs-scenery} follows.

For a $\mu$-typical $x$, suppose that $x$ equidistributes for some measure $\nu$ under $T_n$, along a sequence $N_j$. Since $P$ is nontrivial, Lemma \ref{lem:continuity-of-EFD-measures} applied with $t_0=\log n$, together with the representation \eqref{eq:A-adic-vs-scenery}, imply that the condition \eqref{eq:boundary-condition} in Theorem \ref{thm:local-average} holds. Thus for some sequence $N_j\to\infty$,
\begin{equation} \label{eq:representation-limit-measure}
\nu = \lim_{j\rightarrow\infty}\frac{1}{N_j}\sum_{k=1}^{N_j}c_k\cdot (\delta_{y_k} *\mu_{x,k\log n})|_{[0,1]}\qquad\mbox{weak-* in }\mathcal{P}([0,1])
\end{equation}
Passing to a  further subsequence we may assume that the joint distribution of $c_k,y_k$ and the measures converges, i.e. that $\frac{1}{N_j}\sum_{k=0}^{N_j-1}\delta_{(c_k,y_k,\mu_{x,k\log n})}$ converges to a probability measure $Q$ on $\Omega=\mathbb{R}\times[-1,1]\times\mathcal{P}(\mathcal{P}([-1,1]))$ (to see that the distribution of the $c_k$'s is tight, we use Proposition \ref{pro:large-mass-near-0}). Moreover, thanks to Lemma \ref{lem:discrete-equidistribution-and-spectrum}, the measure marginal of $Q$ is $P$: this is the point of the proof where the spectral assumption is used.

Taking stock, we have proved the following representation of $\nu$.

\begin{thm}\label{thm:integral-representation}
Let $\mu$ be a measure on $[0,1]$ which generates a distribution $P$ at a.e. point, and $\Sigma(P,S)\cap \frac{1}{\log n}\mathbb{Z}=\{0\}$. Then for $\mu$-a.e. $x$, if $x$ equidistributes under $T_n$ for $\nu$ along some subsequence, then there is an auxiliary probability space $(\Omega,\mathcal{F},Q)$ and measurable functions $c:\Omega\to(0,\infty)$, $y:\Omega\to[-1,1]$ and $\eta:\Omega\to\mathcal{P}[-1,1])$, such that $\eta$ is distributed according to $P$, and \[
  \nu = \int c_\omega \cdot (\delta_{y_\omega}*\eta_\omega)|_{[0,1]}\,dQ(\omega).
\]
\end{thm}

Invoking Proposition \ref{pro:positive-dim}, we immediately get:
\begin{cor}\label{cor:dimension-of-limit-measures}
A measure $\nu$ as in the theorem is of dimension at least $\delta$ (the a.s. dimension of measures drawn according to $P$); in particular, $\dim\nu>0$.
\end{cor}

The changes needed to prove this for $T_\beta$ and non-integral Pisot $\beta$ are minimal. In this case one uses the partition $\mathcal{A}$ of $[0,1]$ into intervals $[j/\beta,(j+1)/\beta)\cap[0,1]$. The main difference is that now the identity \eqref{eq:A-adic-vs-scenery} is not always true because the length of $\mathcal{A}^k(x)$ is no longer constant, and so the left hand side of  \eqref{eq:A-adic-vs-scenery} is generally the restriction of the right hand side to a shorter interval (followed by normalization). If in $\eqref{eq:A-adic-vs-scenery}$ we replace the restriction on the right hand side with restriction to the appropriate interval $I_k(x)\subseteq [0,1]$, then we obtain a representation of the same kind as in Theorem \ref{thm:integral-representation} but of the form \begin{equation} \label{eq:integral-representation-Pisot}
  \nu = \int c_\omega \cdot (\delta_{y_\omega}*\eta_\omega)|_{I_\omega}\,dQ(\omega)
\end{equation}
where $I_\omega\subseteq [0,1]$ is a random interval. The missing ingredient in this argument is that a-priori the intervals $I_k$ may be vanishingly short for a positive frequency of $k$, and we must ensure that the distribution of lengths does not concentrate on $0$, i.e. we must ensure that $I_\omega$ is a.s.  of positive length. This is where the Pisot property of $\beta$ comes into play, via

\begin{lem}\label{lem:garsia-simplified}
There is a constant $c>1$ such that for any $k$ and interval $I\in \mathcal{A}^k$, the length of $I$ satisfies  $c^{-1}\beta^{-k}<|I|<c\beta^{-k}$.
\end{lem}
This is a consequence of a classical lemma of Garsia \cite{Garsia62}, stated more completely below, see Lemma \ref{lem:Garsia}. We note that the weaker version stated here continues to hold for the larger class of $\beta$ for which the $\beta$-shift $T_\beta$ satisfies the specification property, but for these numbers the results of the next section do not appear to hold.

\subsection{\label{sub:resonance-and-dissonance}Resonance and dissonance}

As indicated in Section \ref{sub:sketch-of-proof}, the second idea we need for the proof of Theorem \ref{thm:WM-case} is that, among invariant measures for $T_\beta$ of positive dimension, the Parry measure can be identified by the behavior of its dimension under convolutions. Following terminology of Peres and Shmerkin \cite{PeresShmerkin09}, we say that measures
$\mu,\nu\in\mathcal{P}(\mathbb{R})$ \emdef{resonate} if
\begin{equation}
\dim\mu*\nu<\min\{1,\dim\mu+\dim\nu\}\label{eq:dissonance}
\end{equation}
otherwise they \emdef{dissonate}.

As a general rule, measures should dissonate; resonance requires, heuristically, that they have some common structure. This heuristic can be made precise in many ways. For example, as an immediate consequence of Corollary \ref{cor:marstrand}, we have

\begin{thm} \label{thm:marstrand}
If $\mu,\nu$ are Borel probability
measures on $\mathbb{R}$, then for Lebesgue-a.e. $t\in\mathbb{R}$,
the measures $\mu$ and $S_{t}\nu$ dissonate.

Moreover, suppose that $\dim\mu|_I=\dim\mu$ for any interval $I$ of positive $\mu$-measure. Then for a.e. $t$, if $I$ is any set of positive $S_{t}\mu$-measure, then $(S_{t}\mu)|_{I}$ and $\nu$ dissonate.
\end{thm}
\begin{proof}
This is a consequence of Theorem \ref{thm:classical-marstrand}, and elementary properties of $\dim$.
\end{proof}

Unlike the ``generic'' case, where dissonance is the rule, for integer $n$, $T_n$-invariant measures of dimension strictly between $0$ and $1$ do resonate, often with themselves and always with other $T_n$ invariant measures. For example consider a $T_n$-invariant measure $\mu$ with $1/2<\dim \mu<1$. That $\mu$ resonates with itself can be seen as follows. First, $\mu*\mu$ has the same dimension as the dimension of the self-convolution $\nu=\mu*\mu$ with the convolution taken in $\mathbb{R}/\mathbb{Z}$ (this is because the map $\mathbb{R}\to\mathbb{R}/\mathbb{Z}$ is a countable to 1 local isometry). Consider the Fourier transform: $\hat{\nu}(k)=\hat{\mu}(k)^2$. Since $\mu$ is not Lebesgue measure it has a non-zero coefficient, hence so does $\nu$, and therefore $\nu$ is not Lebesgue measure. But it is a well known fact that the only $T_n$-invariant measure of dimension $1$ is Lebesgue measure, and $\nu$ is $T_n$-invariant; hence $\dim\nu=\dim\mu*\mu<1=\min\{1,\dim\mu+\dim\mu\}$.

We will require the following strengthening of the fact above.

\begin{thm}\label{thm:existence-of-resonant-measures}
Let $\beta>1$ be a Pisot number. Then there is a sequence of probability measures $\tau_{1},\tau_{2},\ldots$
on $\mathbb{R}$ with
\[
\dim\tau_n\to 1 \quad\text{as } n\to\infty,
\]
such that any $T_{\beta}$-invariant measure $\nu$
with $0<\dim\nu<1$ resonates with $\tau_{n}$ for all large enough
$n$.
\end{thm}

In order not to interrupt the main line of argument, we postpone the proof to Section \ref{sec:resonance}.

\subsection{\label{sub:proof-of-WM-case}Proof of Theorem \ref{thm:WM-case}}

Let $\beta>1$ be a Pisot number. Let $\mu\in\mathcal{P}([0,1])$ generate an $S$-ergodic and non-trivial distribution $P$, and suppose that $k/\log\beta$ is not in $\Sigma(P,S)$ for any $k\in\mathbb{Z}\setminus\{0\}$.

Let $\nu_\beta$ be the unique absolutely continuous invariant measure for $T_\beta$ (the Parry measure). The following fact is standard, but we include a proof as we have not been able to find a reference.

\begin{lem} \label{lem:unique-measure-of-maximal-dim}
The measure $\nu_\beta$ is also the unique invariant measure of maximal dimension $1$.
\end{lem}
\begin{proof}
It is well known that $\nu_\beta$ is the only measure of maximal entropy $\log\beta$ (\cite{Hofbauer78}, see also \cite[Remark 2.4]{Sidorov03}). Let $\theta\neq\nu_\beta$ be another invariant measure. By the Shannon-McMillan-Breiman applied to the (generating) partition $\{[k/\beta,(k+1)/\beta)\cap[0,1]\}$, and Lemma \ref{lem:garsia-simplified},
\[
\underline{\dim}(\theta,x)\le \lim_{n\to \infty} \frac{\log\theta([x-\beta^{-n},x+\beta^{-n}])}{n\log \beta} = \frac{h(\theta,x)}{\log\beta},
\]
for $\theta$-almost all $x$, where $h(\theta,x)$ is the entropy of the ergodic component of $x$. Since $h(\theta)<h(\nu_\beta)=\log\beta$, there is a set of positive measure where the right-hand side above is $<1$. In light of the characterization of $\dim$ using local dimensions given in Equation \eqref{eq:dim-using-local-dim}, $\dim\theta<1$, as desired.
\end{proof}

Fix a $\mu$-typical $x$. It suffices to show that if $x$ equidistributes under $T_\beta$ along a sub-sequence for a measure $\nu$, then $\nu$ is the unique absolutely continuous $T_\beta$-invariant measure $\nu_\beta$.

From Theorem \ref{thm:integral-representation} (and the discussion following it for the general Pisot case), we have the representation
\[
  \nu = \int c_\omega \cdot (\delta_{y_\omega}*\eta_\omega)|_{I_\omega}\,dQ(\omega)
\]
where $c_\omega,y_\omega,\eta_\omega,I_\omega$ are defined for $\omega$ in some auxiliary probability space $(\Omega,\mathcal{F},Q)$, and the distribution of $\eta_\omega$ is $P$. Recalling Proposition \ref{pro:positive-dim}, let $\delta>0$ denote the a.s. dimension of measures drawn according to $P$, so also $\dim \eta_\omega=\delta$ a.s. In particular, $\dim\nu>0$ and $\nu$ is non-atomic.

\begin{lem}
$\nu$ is $T_\beta$-invariant.
\end{lem}
\begin{proof}
$T_\beta$ has finitely many discontinuities, and $\nu$ is non-atomic, so the set of discontinuities has $\nu$-measure zero. Since $\nu$ arises as the measure for which $x$ equidistributes subsequentially, it is $T_\beta$-invariant.
\end{proof}

\begin{lem} \label{lem:dissonance}
Let $\tau$ be a probability measure on $\mathbb{R}$ with $\dim\tau\geq 1-\delta$. Then $\dim\tau*\eta_\omega=1$ for $Q$-a.e. $\omega$.
\end{lem}
\begin{proof}
Using $S$-invariance of $P$, Fubini and Theorem \ref{thm:marstrand}, \begin{eqnarray*}
  \int \dim(\tau*\eta_\omega)\,dQ(\omega) & = & \int \dim(\tau*\eta)\,dP(\eta) \\
       & = & \int_0^1 \int\dim(\tau*\eta)\,dS_tP(\eta)\,dt\\
       & = & \int \int_0^1\dim(\tau*S_t\eta)\,dt\,dP(\eta)\\
       & = & \int \min\{1,\dim\tau+\dim\eta)\,dP(\eta)\\
       & = & 1.
\end{eqnarray*}
Since the integrand on the left hand side is $\leq1$, it is a.s. equal to $1$, as claimed.
\end{proof}

Now let $\{\tau_n\}$ be the sequence of resonant measures provided by Theorem \ref{thm:existence-of-resonant-measures}. Then $\dim\tau_n\to 1$, so for $n$ large enough we have $\dim \tau_n>1-\delta$, hence by linearity of convolution, basic properties of dimension, and the previous lemma, \begin{eqnarray*}
  \dim \tau_n * \nu & = & \dim \left( \tau_n * \int c_\omega \cdot (\delta_{y_\omega}*\eta_\omega)|_{[0,1]}\,dQ(\omega)\right) \\
                    & = & \dim \left( \int c_\omega \cdot (\tau_n*\delta_{y_\omega}*\eta_\omega)|_{[0,1]}\,dQ(\omega)\right) \\
                    &  \geq & \essinf_{\omega\sim Q} \dim(\tau_n*\delta_{y_\omega}*\eta_\omega|_{[0,1]})\\ \\
                    &  \geq & \essinf_{\omega\sim Q} \dim(\tau_n*\eta_\omega)\\ \\
                    & = & \essinf_{\eta\sim P} \dim(\tau_n*\eta)\\
                    &   =  & 1.
\end{eqnarray*}
But by choice of $\tau_n$, this is possible only if $\dim\nu=0$ or $1$. Since $\dim\nu>0$, we must have $\dim\nu=1$. Lemma \ref{lem:unique-measure-of-maximal-dim} then allows us to conclude that $\nu$ is the Parry measure for $T_\beta$, as desired.

This completes the proof of Theorem \ref{thm:WM-case}

There is a version of Theorem \ref{thm:WM-case} for measures which do not generate a distribution. For a measure $\mu$ and a typical point $x$ let $\mathcal{D}(\mu,x)\subseteq\mathcal{D}$ denote the set of accumulation points of $\frac{1}{T}\int_0^T\delta_{\mu_{x,t}}dt$ as $t\to\infty$. In \cite[Theorem 1.7]{Hochman12b} it was shown that for $\mu$-a.e. $x$, this set consists EFDs. An easy adaptation of the proof of the theorem above shows that if $\mu$ is a measure such that a.s., $\mathcal{D}(\mu,x)$ contains only non-trivial ergodic distributions which do not have $k/\log n$ in their spectrum, then $\mu$ is pointwise $n$-normal. We shall not give the proof of this in detail.

\section{\label{sec:resonance}Construction of resonant measures}

The proof of Theorem \ref{thm:existence-of-resonant-measures} is slightly more transparent in the case that $\beta$ is an integer. After some preliminaries we will prove this case, since it is shorter and may shed light on the general case.

\subsection{\label{sub:entropy}Preliminaries on entropy}

We use standard notation and properties for the entropy $H(\mu,\mathcal{P})$ of a measure $\mu$ with respect to a partition $\mathcal{P}$. See \cite{Walters82} or any textbook in ergodic theory for details.

Let $\mathcal{A}^k$ be the partition of $\mathbb{R}$ into $k$-generation $n$-adic intervals, that is, intervals $[r/n^k,(r+1)/n^k)$ for $r\in\mathbb{N}$. For a $T_n$-invariant measure $\mu$, the Kolmogorov-Sinai entropy is given by \[
   h(\mu) = \lim_{k\to\infty}\frac{1}{k}H(\mu,\mathcal{A}^k)
\]
and the limit is also the infimum. In general, $h(\mu)\le \log n$, with equality if and only if $\mu$ is Lebesgue measure $\lambda$. We also have
\[
   \frac{1}{\log n}h(\mu) \geq \dim \mu
\]
with equality if $\mu$ is ergodic; in general $\dim\mu$ is the essential infimum over the dimensions (=normalized entropies)  of the ergodic components of $\mu$. This follows e.g. from the proof of Lemma \ref{lem:unique-measure-of-maximal-dim}.

The quantity $H(\mu,\mathcal{A}^k)$ is  not continuous in $\mu$, however we have the following approximate continuity under translation: If $\eta$ is a measure supported on an interval of length $<1/n^k$, then
\[
   |H(\eta*\mu,\mathcal{A}^k)-H(\mu,\mathcal{A}^k)|<c
\]
where $c$ is a universal constant.

\subsection{\label{sub:resonance-int-case}The integer case}
Fix an integer $n\geq 2$. Our goal is to construct a sequence of probability measures $\tau_{1},\tau_{2},\ldots$
on $\mathbb{R}$ such that $\dim\tau_i\to 1$ and any $T_{n}$-invariant measure $\nu$
with $0<\dim\nu<1$ resonates with $\tau_{i}$ for all large enough $i$.

We will use the standard identification of the map $T_n$ on $[0,1]$ with the shift map on the sequence space $\{0,\ldots,n-1\}^\mathbb{N}$, given by the base-$n$ expansion. This is defined uniquely off a countable set of points and hence for non-atomic measures is an a.e. isomorphism, so we will not distinguish between the models.

Let $N$ be an integer and define a measure $\nu_N$ on infinite sequences of digits $\{0,\ldots,n-1\}$ as follows. Set the first $N$ digits to be $0$. Let the next $N^2$ digits be chosen independently
and equiprobably from $\{0,\ldots,n-1\}$. Repeat this procedure, independently of previous choices,
for each subsequent block of $N+N^2$ symbols. Write $\nu_N$ also for the corresponding measure on $[0,1]$. Now, this measure is not $T_n$-invariant but it is $T_n ^{N+N^2}$-invariant, so the measure \[
   \tau_N=\frac{1}{N+N^2}\sum_{i=0}^{N+N^2-1}T_{n}^{i}\nu_N
\]
is $T_n$-invariant.

It is elementary to use Equation \eqref{eq:dim-using-local-dim} to show that $\dim \nu_N=N^2/(N+N^2)$, and so the same is true for $T^i \nu_N$, and hence for $\tau_N$. Thus $\dim\tau_N\to 1$ as $N\to\infty$.

Now let $\mu$ be a $T_n$-invariant measure and suppose that it is not Lebesgue measure. We aim to show that  $\dim\tau_N*\mu< 1$ for large enough $N$. Using the $T_n$-invariance of $\mu$ and the fact that $T_n^i$ is piecewise affine with constant expansion, we have \begin{eqnarray*}
  \dim \tau_N*\mu & = &\dim \left(\frac{1}{N+N^2}\sum_{i=0}^{N+N^2-1}(T_n^i\nu_N)*\mu\right)\\
                  & = &\inf_{0\leq i<N+N^2} \dim (T_n^i\nu_N)*\mu\\
                  & = &\inf_{0\leq i<N+N^2} \dim (T_n^i\nu_N)*(T_n^i\mu)\\
                  & = &\inf_{0\leq i<N+N^2} \dim (\nu_N*\mu)\\
                  & = & \dim \nu_N*\mu.
\end{eqnarray*}
Thus it is enough to show that $\dim(\nu_N*\mu)<1$ for large enough $N$, and since $\nu_N*\mu$ is $T_n^{N+N^2}$-invariant, we only need to show that $\nu_N*\mu$ is not Lebesgue.  $\nu_N$ is concentrated on the interval $[0,n^{-N})$, so  we know that \[
   H(\nu_N*\mu,\mathcal{A}^N) < H(\mu,\mathcal{A}^N) + c
\]
where $c$ is a universal constant. Since $\mu$ is not Lebesgue, it has less than full entropy, and hence $H(\mu,\mathcal{A}^N)<(1-\e)N\log n$ for some $\e>0$ independent of $N$. Thus
\[
 H(\nu_N*\mu,\mathcal{A}^N) < (1-\e)N\log n+c< N\log n
\]
for large enough $N$. Dividing by $N$ and taking the infimum over $N$ we find that $h(\tau_N*\nu)<\log n=h(\lambda)$, where $\lambda$ is Lebesgue measure, so $\tau_N*\nu\neq\lambda$, as desired.

\subsection{Dynamics of beta transformations} \label{sub:beta-transformations}

We review some basic facts about the beta transformations $T_\beta:x\mapsto \beta x\bmod 1$. We refer the reader to the surveys \cite{Blanchard89, Sidorov03} for further information and references.

Recall that $[m]=\{0,\ldots,m-1\}$. For each $\beta>1$, there is a $T$-invariant closed subset $X_\beta\subseteq [\lceil \beta\rceil]^\N$ (known as the \emdef{$\beta$-shift}) such that the beta expansion map $\pi:X_\beta \to [0,1]$, $\pi(x)=\sum_{n=1}^\infty x_n \beta^{-n}$ semi-conjugates the action of the shift map $T$ on $X_\beta$ with the action of $T_\beta$ on $[0,1]$. Further, $\pi$ is injective on $X_\beta$,  except at countably many points on which it is two-to-one.  In particular, any non-atomic $T_\beta$ invariant measure lifts uniquely to a shift-invariant measure on the $\beta$-shift.

We will require the following lemma on the structure of the $\beta$-shift for Pisot $\beta$.

\begin{lem} \label{lem:structure-Pisot-shift}
Let $\beta$ be a Pisot number. There exists $N_0=N_0(\beta)\in\N$ with the following property: let $\{ x_i\}$ be finite words in $X_\beta$ (i.e. $X_\beta$ contains infinite words starting with each of the $x_i$). Then the infinite concatenation $(0^N x_1 0^N x_2 \ldots)$ is in $X_\beta$ for $N> N_0$.
\end{lem}
\begin{proof}
The following characterization of $X_\beta$ is essentially due to Parry \cite{Parry60}, see also \cite[Proposition 2.3]{Blanchard89}.  Let $a$ be the lexicographically least $\beta$-expansion of $1$, i.e. the lexicographically smallest sequence $a\in [\lceil \beta\rceil ]^\N$ such that $1=\sum_{i=1}^\infty a_i \beta^i$. Then $x\in X_\beta$ if and only if $T^k x\prec a$ for all $k$, where $\prec$ denotes lexicographically smaller or equal.

On the other hand, if $\beta$ is Pisot, then the sequence $a$ is eventually periodic, see \cite[Section 4.1]{Blanchard89}. It cannot end in infinitely many zeros because $ (a_1\ldots a_{k-1} (a_k-1))^\infty\prec (a_1\ldots a_k 0^\infty)$ and both sequences represent the same number in base $\beta$. It follows that the number of consecutive zeros in $a$ is bounded by some integer $N_0$. But then it is clear that for any finite words $\{y_i\}$ in $X_\beta$, any $N> N_0$ and any $\ell\ge 0$, we have $(0^\ell y_1 0^N y_2 0^N\ldots)\prec a$. This gives the claim.
\end{proof}

\subsection{Resonance in the Pisot case}

The proof of the Pisot case is not unlike the integer one. The main difference is that convolutions of $T_\beta$-invariant measures are no longer invariant or related in any obvious way to an invariant measure. This makes estimating their dimension more involved.

For the rest of this section we fix a Pisot number $\beta>1$, and write $B=\lceil\beta\rceil$. Given an integer $D\ge B$ (often implicit), and a $[D]$-valued finite or infinite sequence $x$ of length $|x|$, we let $\pi$ be the $\beta$ expansion map, i.e. $\pi(x) =\sum_{k=1}^{|x|} x_k\,\beta^{-k}$. If $|x|=\infty$, we also write $\pi_k(x)=\pi(x|_{\{1,\ldots,k\}})$. A key role will be played by the following partition of $D^k$:
 \[
 \mathcal{P}_k = \{ \pi^{-1}(\pi x): x\in [D]^k\}.
 \]
 The property of Pisot numbers that will be used in the proof is given in the following classical Lemma of Garsia \cite[Lemma 1.51]{Garsia62}:
 \begin{lem} \label{lem:Garsia}
 There exists $c>0$ (depending on $\beta$ and $D$) such that for any $x,y\in [D]^k$, either $\pi(x)=\pi(y)$, or $|\pi(x)-\pi(y)|\ge c \beta^{-k}$.
 \end{lem}

We quote a basic fact for later reference:

\begin{lem} \label{lem:comp-discrete-measure}
Let $\widetilde{\mu}$ be any measure on $[D]^\N$, and set
\[
a = a_{\beta,D}= \frac{(D-1)\beta^{-1}}{1-\beta^{-1}}.
\]
Then for any set Borel $A\subseteq\R$ and any $k\in\N$,
\begin{enumerate}
\item $\pi\widetilde{\mu}(A) \le  \pi_k\widetilde{\mu}(A^{(a\beta^{-k})})$,
\item $\pi\widetilde{\mu}(A^{(a\beta^{-k})}) \ge \pi_k\widetilde{\mu}(A)$,
\end{enumerate}
where $A^{(\delta)}$ denotes the $\delta$-neighborhood of $A$.
\end{lem}
\begin{proof}
Immediate from the fact that if $x\in [D]^\N$, then
\[
|\pi(x)-\pi_k(x)| \le \sum_{i=k+1}^\infty (D-1)\beta^{-i} = a\,\beta^{-k}.
\]
\end{proof}

The following lemma is similar to \cite[Lemma 3]{Lalley98}.

\begin{lem} \label{lem:entropy-dim}
Let $\widetilde{\mu}$ be an $T$-invariant measure on $[D]^\N$ (as before $T$ is the shift map). Then
\[
\dim\pi\widetilde{\mu} \le \lim_{k\to\infty} \frac{ H(\widetilde{\mu},\mathcal{P}_k)}{k\log\beta}= \inf_{k\ge 1}\frac{H(\widetilde{\mu},\mathcal{P}_k)}{k\log\beta}.
\]
\end{lem}
\begin{proof}
Write $\mu=\pi\widetilde{\mu}$. For the first inequality, note first that
\[
\mu(B(\pi x,(1+a)\beta^{-k})) \ge \widetilde{\mu}(\mathcal{P}_k(x)),
\]
for any $x\in [D]^\N$, where $a$ is the constant from Lemma \ref{lem:comp-discrete-measure}. The inequality follows by combining this and Fatou's lemma applied to the sequence
\[
g_k(x) = \frac{\log\widetilde{\mu}(\mathcal{P}_k(x))}{-k\log\beta}.
\]

For the second equality, it is enough to show that the sequence $H(\mathcal{P}_k,\widetilde{\mu})$ is sub-additive. The partition $\mathcal{P}_k \vee T^{-k}\mathcal{P}_m$ is a refinement of $\mathcal{P}_{m+k}$, since $\pi_k(x)$ and $\pi_m(T^k x)$ determine $\pi_{m+k}(x)$. Thus
\begin{align*}
H(\mathcal{P}_{m+k},\widetilde{\mu}) &\le H(\mathcal{P}_k \vee T^{-k}\mathcal{P}_m,\widetilde{\mu})\\
&\le H(\mathcal{P}_k,\widetilde{\mu}) + H(\mathcal{P}_m,\widetilde{\mu}),
\end{align*}
using the invariance of $\mu$.
\end{proof}

Note that in the above we do not assume that $\pi\widetilde{\mu}$ is $T_\beta$-invariant.

\begin{lem} \label{lem:entropy-dim-drop}
Let $\widetilde{\mu}$ be the lift to $X_\beta \subseteq [B]^\N$ of a non-atomic $T_\beta$-invariant measure $\mu$. If $\mu$ is not the Parry measure, then
\[
\lim_{k\to\infty} \frac{ H(\mathcal{P}_k,\widetilde{\mu})}{k} < \log\beta.
\]
\end{lem}
\begin{proof}
We claim that the limit in the left-hand side equals the entropy of $\mu$ under $T_\beta$; this will imply the lemma since the Parry measure is the unique measure of maximal entropy $\log\beta$.

As before, let $\mathcal{A}$ be the partition of $[0,1]$ into intervals $[j/\beta,(j+1)/\beta)\cap[0,1]$, and $\mathcal{A}^k = \mathcal{A}\vee\cdots \vee T_\beta^{k-1} \mathcal{A}$. Let also $\mathcal{Q}_k$ be the partition of $[0,1]$ into half-open intervals determined by the points $\{ \pi(x):x\in D^n\}$. Since $\widetilde{\mu}$ is supported on $X_\beta$, $H(\mathcal{P}_k,\widetilde{\mu})=H(\mathcal{Q}_k,\mu)$ (the correspondence between the elements of both partitions follows from the fact that $X_\beta$ is composed of the lexicographically least sequences with a given $\beta$ expansion, which implies tha the lexicographic order on $X_\beta$ projects onto the usual order of $[0,1]$, see e.g. \cite{Sidorov03}).

On the other hand, it is easy to see that $\mathcal{Q}_k$ refines $\mathcal{A}^k$ and, thanks to Garsia's Lemma, each atom of $\mathcal{A}^k$ is the union of a uniformly bounded number of atoms of $\mathcal{Q}_k$. Hence
\[
\lim_{k\to\infty} \frac1k H(\mathcal{P}_k,\widetilde{\mu}) = \lim_{k\to\infty} \frac1k H(\mathcal{A}^k,\mu) = h(\mu),
\]
as claimed.
\end{proof}

\begin{proof}[Proof of Theorem \ref{thm:existence-of-resonant-measures}]
Let $M \gg N \gg 1$  be large numbers; $M$ will be chosen as a function of $N$ later. We construct a measure $\tau=\tau_{M,N}$ on $[B]^\N$ as follows. Let  $\widetilde{\lambda}$ be the lift of the Parry measure $\lambda_\beta$ to the code space. Now let $\widetilde{\tau}_0$ be the measure on $[B]^\N$ defined as follows (compare with the measure constructed in the integer case). The first $N$ digits are $0$. The next $M$ digits are chosen according to $\widetilde{\lambda}$. Continue this procedure for each block of $N+M$ digits, with all the choices independent.

As in the integer case, this measure is $T^{M+N}$-invariant but not $T$-invariant, so we define
\[
\widetilde{\tau} = \widetilde{\tau}_{M,N}=\frac{1}{M+N} \sum_{i=0}^{M+N-1} T^i\widetilde{\tau}_0,
\]
which is shift-invariant and ergodic. Lemma \ref{lem:structure-Pisot-shift} shows that, provided $N$ is large enough,  $\widetilde{\tau}_0$ and hence also $\widetilde{\tau}$ are defined on the $\beta$-shift $X_\beta$. In particular $\tau=\tau_{N,M}:=\pi\widetilde{\tau}$ is $T_\beta$-invariant.

Let $\tau_0=\pi\widetilde{\tau}_0$. The Parry measure $\lambda_\beta$ has a bounded density with respect to Lebesgue measure (in the Pisot case it is actually piecewise constant). It follows that if $I$ is an interval determined by two consecutive points of the form $\sum_{i=1}^{(N+M)k} x_i\beta^{-i}$, then $\tau_0(I) \le c^k\beta^{-Mk}$, where $c>0$ is a constant that depends only on $\beta$ (in particular, it is independent of $M$, $N$ and $I$). By Garsia's Lemma \ref{lem:Garsia}, any interval of length $2\beta^{-(M+N)k}$ can be covered by a uniformly bounded number of such $I$, and we conclude that
\[
\liminf_{r\downarrow 0}\frac{\log\tau_0([x-r,x+r])}{\log r} \ge \frac{\log c+M\log\beta}{(M+N)\log\beta}.
\]
Thus for any $N$, by taking $M=M(N,c)$ large enough, we can ensure that $\dim\tau=\dim\tau_0>1-1/N$.

It remains to show that if $N$ is large enough, then for any $M$, $\dim(\mu*\tau)<1$. Since $\mu$ is invariant, arguing as in the integer case we see that it suffices to show this with $\tau_0$ in place of $\tau$ (note that the argument does not use invariance of the convolved measure, only the identity $\dim(T_\beta^i\mu*T_\beta^i\nu)=\dim(\mu*\nu)$, which holds for any map that is piecewise affine with constant slope, in particular $T_\beta$).

Note that $\mu*\tau_0$ is the projection of $\mu\times \tau_0$ under the addition map $(x,y)\to x+y$, and hence $\mu*\tau_0 = \pi\widetilde{\rho}$, where $\widetilde{\rho}$ is the image of $\widetilde{\mu}\times\widetilde{\tau}$ on $[B]^\N\times [B]^\N$ under the map $(x,y)\to (x_i+y_i)_i$ (so that $\widetilde{\rho}$ is defined in $[2B-1]^\N$, and it is $T$-invariant). It follows from Lemma \ref{lem:entropy-dim} (applied with $D=2B-1$ and the partitions $\mathcal{P}_k$ defined in terms of $D$) that
\begin{equation} \label{eq:upper-bound-entropy-dim}
\dim(\mu*\tau_0) \le \frac{H(\mathcal{P}_{M+N},\widetilde{\rho})}{(M+N)\log\beta}.
\end{equation}

Since, by assumption, $\dim\mu<1$, we know from Lemma \ref{lem:entropy-dim-drop} that there is $\e>0$ such that, if $N$ is large enough, then
\[
H(\mathcal{P}_N,\widetilde{\mu}) < (1-\e) N\log\beta.
\]

Using this, the fact that $\mathcal{P}_N\vee T^{-N}\mathcal{P}_M$ refines $\mathcal{P}_{M+N}$, that $|\mathcal{P}_M|\le C\,\beta^M$ (by Garsia's Lemma), and that $\tau$ is concentrated on $\{ x\in [B]^\N: x_1=\cdots=x_N=0\}$ (which implies that $\widetilde{\rho}$ and $\widetilde{\mu}$ coincide on $\mathcal{P}_N$), we estimate
\begin{align*}
H(\mathcal{P}_{M+N},\widetilde{\rho}) &\le H(\mathcal{P}_N,\widetilde{\rho}) + H(T^{-N}\mathcal{P}_M,\widetilde{\rho})\\
&\le H(\mathcal{P}_N,\widetilde{\rho}) + \log|\mathcal{P}_M|\\
&\le H(\mathcal{P}_N,\widetilde{\mu}) + M\log\beta + \log C\\
&\le ((1-\e)N+M+\log C)\log\beta.
\end{align*}
Recalling \eqref{eq:upper-bound-entropy-dim}, we conclude that there is $N_0$ such that for all $N\ge N_0$ and all $M\in \N$,
\[
\dim(\mu*\tau_0)< 1.
\]
This completes the proof.
\end{proof}

\section{\label{sec:app-to-fractals-1}Application to iterated function systems: Theorems \ref{thm:dissonant-IFSs}, \ref{thm:totally-nonlinear-IFSs} and \ref{thm:BA-normal-numbers}}

\subsection{\label{limit-geometries}Limit geometries}

In this section we fix the following notation. Let $\mathcal{I}=\{f_0\,\ldots\,f_{r-1}\}$ be an IFS on an interval which, without loss of generality, we assume is $[0,1]$.  We will henceforth assume that $\mathcal{I}$ is $C^\alpha$ for some $\alpha>1$ or $\alpha=\omega$, and regular as defined in the introduction.

Let $\mu$ be a quasi-product measure for $\mathcal{I}$. The next lemma contains the key structural information we shall require about $\mu$; it is a manifestation of ideas that go back to Sullivan \cite{Sullivan88}. We write $\nu_1\sim_C\nu_2$ to denote that the measures $\nu_1,\nu_2$ are mutually absolutely continuous with both Radon-Nikodym densities bounded by $C$, i.e. $1/C\leq d\nu_1/d\nu_2 \leq C$.

\begin{lem} \label{lem:limit-geometries}
Then there is $C=C(\mu)>0$ such that the following holds. Let $x\in\supp\mu$ and let $\nu$ be an accumulation point of $\mu_{x,t}$ as $t\to\infty$. Then $\nu\sim_C (g\mu)|_{[-1,1]}$ and $\mu\sim_C (h\nu)|_{[0,1]}$ for some $g,h\in\diff^\alpha(\R)$.
\end{lem}
\begin{proof}
Given a finite sequence $y\in [r]^n$, let $f_y = f_{y_1}\circ \cdots \circ f_{y_n}$ and $f_y^*=A_y f_y$, where $A_y$ is the renormalizing homothety mapping $f_y([0,1])$ back to $[0,1]$. It is proved in \cite[Theorems 5.9 and 6.1]{BedfordFisher97} that  for a left-infinite sequence $y=(y_i)_{i=-\infty}^0$, the sequence $f_{y_{-n}\ldots y_{0}}^*$ converges, in the $C^1$ topolofy, to a $C^\alpha$ diffeomorphism $F^*_y$ (these are known as \emph{limit diffeomorphisms}). Moreover, the dependence of $F^*_y$ on $y$ is uniformly continuous. In particular, the family $\{ f^*_y\}$, where $y$ ranges over all finite words, is relatively compact in the $C^1$ topology.

For the first part, $\nu\sim_C (g\mu)|_{[-1,1]}$, one notes that $\mu_{x,t}$ is $C$-equivalent to a bounded translation, a restriction and normalization of $S_s f^*_y\mu$  for an appropriate word $y=y(t)$, whose length tends to $\infty$ with $t$, and some $s=s(t)\in [0,L]$, where $L$ depends only on the IFS (one can take $L$ to be the maximum of $|\log f'_i(x)|$ over $i\in [r]$ and $x\in [0,1]$). Also, the space of measures $C$-equivalent to $\mu$ is weak-* closed. Thus, up to passing to a subsequence,  $\nu$ is $C$-equivalent to a translation, restriction and normalization of $S_s F\mu$ for some limit diffeomorphism $F$ and $s\in [0,L]$.

For the second statement, note that it follows from the first part that $\nu|_I \sim_C g(\mu|_J)$ for suitable intervals $I,J$. Moreover, we can take $J=f_y([0,1])$ for some word $y$. In this case $\mu|_J \sim_C f_y\mu$ by the quasi-product property, so the claim follows with $h=(g f_y)^{-1}$ (and a different value of $C$).
\end{proof}

Observe that by Lemmas \ref{lem:image-of-USM} and \ref{lem:limit-geometries}, if $\nu$ is any accumulation point of $\mu_{x,t}$, and $\nu$ generates $P$, then so does $\mu$. This fact will be key in the proof of the following theorem.

\begin{thm}\label{thm:gibbs-generate-EFDs}
$\mu$ generates an $S$-ergodic, non-trivial distribution $P$. Furthermore, there is $C>0$, such that $P$-a.e. $\nu$ satisfies $\nu\sim_C (g\mu)|_{[-1,1]}$ for some $g\in\diff^\alpha(\R)$. .
\end{thm}
\begin{proof}
The proof of the first part is essentially identical to \cite[Proposition 1.36]{Hochman12b}. We sketch the details for completeness.

For $\mu$-a.e. $x$, if the scenery at $x$ equidistributes for $P$ along a subsequence of times $T_k\to\infty$, then $P$ is $S$-invariant and $S$-quasi-Palm \cite[Theorem 1.7]{Hochman12b}. By the ergodic theorem and the $S$-quasi-Palm property, $P$-almost all measures $\nu$ generate the $S$-ergodic component $P_\nu$ of $\nu$ (this argument holds for general measures $\mu$ with no additional assumptions).

Let $x$ be a $\mu$-typical point, and let the scenery at $x$ equidistribute for $P$ along a subsequence (such subsequence exists by compactness). By the previous paragraph, a $P$-typical measure $\nu$ generates an $S$-ergodic distribution $P_\nu$. By the remark preceding the theorem this means that $\mu$ generates $P_\nu$, and the generated distribution is $S$-ergodic. The second part is just Lemma \ref{lem:limit-geometries} and the fact that $P$ is supported on accumulation points of sceneries of $\mu$.
\end{proof}

\subsection{\label{sub:pnt-normality-for-quasi-product}Proofs of Theorems \ref{thm:dissonant-IFSs}, \ref{thm:totally-nonlinear-IFSs} and \ref{thm:BA-normal-numbers}}

\begin{proof}[Proof of Theorem \ref{thm:dissonant-IFSs}]
Let $\mu$ be a quasi-product measure for a $C^{1+\varepsilon}$-IFS $\mathcal{I}$ such that $\lambda(f)\not\sim\lambda(g)$ for some $f,g\in\mathcal{I}$. We want to show that $\mu$ is pointwise $\beta$-normal for any Pisot $\beta>1$. We have already established in Theorem \ref{thm:gibbs-generate-EFDs} that $\mu$ generates an EFD $P$, and our aim is to apply Theorem \ref{thm:WM-pisot-case} to $\mu$, so we must show that $\Sigma(S,P)\cap\frac{1}{\beta}\mathbb{Z}=\{0\}$. To this end, we shall show that if $k/\log\beta\in\Sigma(P,S)$, then $\lambda(f)\sim\beta$ for all $f\in\mathcal{I}$.

We argue by contradiction. Let $t_0=k/\log\beta\in \Sigma(P,S)$ such that $\lambda(f)\nsim \beta$ for some $f\in\mathcal{I}$. Hence
\begin{equation} \label{eq:nontrivial-rotation}
e(t_0\log \lambda(f))\neq 1.
\end{equation}

Let $\nu$ be a $P$-typical measure; we know from Proposition \ref{pro:synchronization} that the phase measure $\theta_\nu$ of the eigenvalue corresponding to $t_0$ is (well defined, and) a single atom. Now by Lemma \ref{lem:limit-geometries}, $\nu$ is also (the restriction of) a quasi-product measure for a conjugated $C^\alpha$ IFS $\mathcal{J}=  g\mathcal{I} =\{g f g^{-1}:f\in\mathcal{I}\}$. Since $\lambda(g f g^{-1})=\lambda(f)$, we may assume without loss of generality that, already for the original measure $\mu$, the phase measure is an atom, say $\delta_z$.

Let $x_0$ be the fixed point of $f$ (this is in the support of $\mu$). Let $U$ be a small interval centered at $x_0$. Since $\mu_U\ll \mu$ and $f(\mu_U)\ll \mu$, by the first part of Corollary \ref{cor:phase-shift}, the phase measures of $\frac{1}{\mu(U)}\mu|_U$ and $\nu_U:=\frac{1}{\mu(U)}f(\mu|_U)$ are (well defined and) equal to $\delta_z$. By the second part of Corollary \ref{cor:phase-shift}, the phase measure of $\nu_U$ equals $\frac{1}{\mu(U)}\int_U \delta_{e(-t_0 \log f'(x))z} d\mu(x)$. Since $f'$ is continuous, this shows that as the size of $U$ tends to $0$, the support of $\theta_{\nu_U}$ tends to $\{e(-t_0 \log\lambda(f)) z\}$. In light of \eqref{eq:nontrivial-rotation}, this is a contradiction, as desired.
\end{proof}

We sketch an alternative proof of Theorem \ref{thm:dissonant-IFSs} which does not directly use EFDs and instead relies on the results from \cite{HochmanShmerkin12} on dissonance between quasi-product measures for regular IFSs. The overall strategy is the same, with the main difference coming in the part that establishes dissonance. Namely, for a $\mu$-typical $x$, suppose $x$ equidistributes for $\nu$ under $T_\beta$ along a sequence $N_k\to\infty$. Using Lemma \ref{lem:limit-geometries}, one can check that a representation \eqref{eq:integral-representation-Pisot} still holds, except that a priori we do not know that the measure marginal of $Q$ is $P$; however, it is easily seen to be supported on limits of sceneries of $\mu$ which, as we know from Lemma \ref{lem:limit-geometries}, are restrictions of quasi-product measures for a smoothly conjugated IFS. Now since the measures $\tau_n$ constructed in  Theorem \ref{thm:existence-of-resonant-measures} are (convex combinations of) quasi-product measures for a homogeneous affine IFS with contraction ratio $\sim\beta$, it follows from \cite[Theorem 1.4]{HochmanShmerkin12} that the measures $\eta_\omega$ in \eqref{eq:integral-representation-Pisot} dissonate with the $\tau_n$ (this is the step that uses that $\lambda(f)\nsim \beta$ for some $f\in\mathcal{I}$), and hence so does $\nu$ if $\dim(\tau_n)$ is large enough. This contradicts Theorem \ref{thm:existence-of-resonant-measures} unless $\nu$ is the Parry measure.

\begin{proof}[Proof of Theorem \ref{thm:totally-nonlinear-IFSs}]
Let $\mu$ be a quasi-product measure for a real-analytic, totally non-linear IFS $\mathcal{I}$. We know that $\mu$ generates an EFD $P$; we will again show that $P$ is weak-mixing, i.e. $\Sigma(P,S)=\{0\}$. Together with Theorem \ref{thm:WM-pisot-case}, this will yield the result.

Suppose for contradiction that $0\neq t_0\in\Sigma(P,S)$. We follow the scheme of the proof of Theorem \ref{thm:dissonant-IFSs}: instead of working with the original measure $\mu$, we consider a $P$-typical measure $\nu$ such that the phase measure is an atom $\delta_z$. This measure is a restriction of the attractor of a conjugated IFS $g\mathcal{I}$, which is also real-analytic. Since $\mathcal{I}$ is totally non-linear, $g\mathcal{I}$ is nonlinear, hence it contains a non-affine analytic map $h$. We can then find a non-trivial interval $U$ meeting the support of $\nu$, on which $h'$ is strictly monotone (this is the point where analyticity gets used; if the IFS was merely $C^2$,  a priori $h'$ might have no point of strict monotonicity on the Cantor set $\supp\nu$). Now arguing as in the proof of Theorem \ref{thm:dissonant-IFSs}, on one hand the phase measure of $\frac{1}{\nu(U)}h(\nu|_U)$ is $\delta_z$, and on the other hand it equals $\frac{1}{\nu(U)}\int_U \delta_{e(-t_0\log h'(x)) z} d\nu(x)$. The latter measure clearly cannot be atomic, so we have reached the desired contradiction.
\end{proof}

The facts on the spectrum $\Sigma(P,S)$ that emerged in the above proofs may find other applications, so we summarize them below.

\begin{thm}\label{thm:spectrum-of-quasi-product}
Let $\mathcal{I}$ be a $C^{1+\e}$ IFS, $\mu$ a quasi-product measure for it, and  $P$ the distribution generated by $\mu$.
\begin{enumerate}
\item Suppose that  $\lambda(f)\nsim t_0$ for some $f\in\mathcal{I}$. Then $k/\log t_0\notin\Sigma(P,S)$ for any $k\in\mathbb{Z}\setminus\{0\}$. In particular, if $\lambda(f_1)\nsim \lambda(f_2)$ for $f_1,f_2\in\mathcal{I}$, then $\Sigma(P,S)=\{0\}$.
\item If $\mathcal{I}$ is $C^\omega$ and totally non-linear or, more generally, if $\mathcal{I}$ has the property that for any limit diffeomorphism $g$, the conjugated IFS $g\mathcal{I}$ contains a map $h$ such that $h'$ is a local diffeomorphism, then $\Sigma(P,S)=\{0\}$.
\end{enumerate}
\end{thm}

To conclude this section, we present the deduction of Theorem \ref{thm:BA-normal-numbers} from Theorem \ref{thm:WM-pisot-case}.

\begin{proof}[Proof of Theorem \ref{thm:BA-normal-numbers}]
Let $a,b$ be two distinct elements of $\Lambda$. Write $x_i$ for the fixed point of the inverse branch of the Gauss map $f_i(x)=1/(x+i)$. Then $x_a$ and $x_b$ are quadratic numbers generating distinct quadratic fields. It follows that $\lambda(f_a^2)=x_a^4\nsim x_b^4=\lambda(f_b^2)$. Hence for any Pisot $\beta>1$, either $\beta\nsim \lambda(f_a^2)$ or $\beta\nsim\lambda(f_b^2)$; by Theorem \ref{thm:dissonant-IFSs}, any quasi-product measure on $C_\Lambda$ is pointwise $\beta$-normal.
\end{proof}

\section{\label{sec:refinements}\label{sec:remaining-apps}A refinement of Theorem \ref{thm:WM-case} and applications}

\subsection{\label{sub:setup-of-refinements}Relaxing the spectral hypothesis}

For an integer $n$, any $T_n$-invariant and ergodic measure $\mu$ generate an EFD $P$, see \cite{Hochman12}.  This $P$ can be rather explicitly described, and its spectrum can be shown to contain non-zero integer multiples of  $\frac{1}{\log m}$ only if either $m\sim n$ or $\log n/\log m\in\Sigma(T,\mu)$. Thus in many cases the pointwise $m$-normality of $\mu$ follows directly from Theorem \ref{thm:WM-pisot-case}. In order to deal with the remaining cases we now present some refinements of Theorem \ref{thm:WM-pisot-case}, in which $k/\log\beta$ is present in the spectrum of $P$, but instead we assume that the phase is ``sufficiently spread out''. We give two versions, the first being simpler to state:

\begin{thm}\label{thm:refinement}
Let $\beta>1$ be a Pisot number. Let  $\mu\in\mathcal{P}([0,1])$ and suppose that $\mu$ generates an $S$-ergodic and non-trivial distribution $P$ which is not $S_{\log\beta}$-ergodic (so that $k/\log\beta\in \Sigma(P,S)$). Further, assume that $\mu$ $\log\beta$-generates an $S_{\log\beta}$-ergodic distribution $P_x$ at $\mu$-a.e. point $x$. Let $\theta=\theta_\mu$ denote the associated phase measure as described in Section \ref{sub:phase}. If $\dim\theta=1$,  then $\mu$ is pointwise $\beta$-normal.
\end{thm}

One consequence is that if $\int P_xd\mu(x)$ is $S$-invariant, then $\mu$ is pointwise $\beta$-normal, as it is clear that in this case the phase measure is invariant under rotations of the circle hence is normalized length measure. Although the theorem above is strong enough for applications, the proofs become simpler using the following variant:

\begin{thm}\label{thm:refinement-2}
Let $\beta>1$ be a Pisot number. Let $\{\mu_\omega\}_{\omega\in\Omega}\subseteq\mathcal{P}(\mathbb{R})$ be a measurable family defined on  a probability space $(\Omega,\mathcal{F},Q)$. Suppose that there is an  $S$-ergodic and non-trivial distribution $P$, which is not $S_{\log\beta}$-ergodic, and such that $Q$-a.e. $\mu_\omega$ generates $P$ and at a.e. point $\log\beta$-generates an $S_{\log\beta}$-ergodic distribution. Let $\theta_{\mu_\omega}$ denote the associated phase measures and $\theta=\int \theta_{\mu_\omega}\,dQ(\omega)$ the ``cumulative'' phase measure. If $\dim\theta=1$,  then $\mu_\omega$ is pointwise $\beta$-normal for $Q$-a.e. $\omega$, and hence also $\mu=\int\mu_\omega\,dQ(\omega)$ is pointwise $\beta$-normal.
\end{thm}

It is clear that the first theorem follows from the second by taking $\mu_\omega=\mu$ for all $\omega$. Nevertheless we shall prove the first, and then explain the changes needed for the second. The proof of Theorem \ref{thm:refinement} follows the scheme of the proof of Theorem \ref{thm:WM-pisot-case} detailed in Section \ref{sub:sketch-of-proof}, with a minimal change to the first step and a more significant change in the proof of the third step, in particular making use of the stronger version of Marstrand's projection theorem given in Theorem \ref{thm:marstrand-dimofexceptions}.

For the rest of this section, suppose that $\beta$ and  $\mu$ are as in the statement of Theorem \ref{thm:refinement}. In particular, let $\theta=\theta_\mu$ be the associated phase measure  with respect to an appropriate eigenfunction $\varphi$ of $(P,S)$, as in Section \ref{sub:phase}. Fix a $\mu$-typical $x_0$ for which $P_{x_0}$ is defined. For $\mu$-typical $x$, define a function $\ell(x)\in [0,1)$ by $\varphi_\mu(x)=e(\ell(x))\varphi_\mu(x_0)$. It follows from the fact that $P=\int_0^{\log\beta} S_t P_x dt$ and the eigenfunction property that $P_x = S_{\ell(x)} P_{x_0}$. Since $\ell(x)$ depends only on $\varphi_\mu(x)$, we will also denote $\ell(z)=\ell(x)$ where $z=\varphi_\mu(x)$, or in other words $e(\ell(z))=z/\varphi_\mu(x_0)$. In particular, $\dim(\ell\theta)=\dim\theta=1$.

Let $\delta$ denote the almost-sure dimension of measures drawn from $P$; it is also the a.s. dimension of measures drawn from $P_x$ for $\mu$-almost all $x$. Recall from Proposition \ref{pro:positive-dim} that $\delta>0$. The following is a refined version of Lemma \ref{lem:dissonance}.

\begin{lem}\label{lem:refined-dissonance}
Let $\tau$ be a probability measure on $\mathbb{R}$ with $\dim\tau\geq 1-\delta$. Then $\dim\tau*\eta=1$ for $\mu$-a.e. $x$ and $P_x$-a.e. $\eta$.
\end{lem}
\begin{proof}
Using Fubini, the fact that $\dim(\ell\theta)=1$, and Theorem \ref{thm:marstrand-dimofexceptions},
\begin{align}
    \int \int \dim(\tau * \eta) \,dP_x(\eta)\,d\mu(x) & =  \int \int \dim(\tau * \eta) \,dS_{\ell(x)}P_{x_0}(\eta) \,d\mu(x) \label{eq:refinement}\\
                 & =  \int \int \dim(\tau * \eta) \,dS_{\ell(z)}P_{x_0}(\eta) \,d\theta(z)\nonumber\\
                 & =  \int \int \dim(\tau * S_{\ell(z)}\eta) \,d\theta(z)\,dP_{x_0}(\eta)\nonumber \\
                 & =  \int \int \dim(\tau * S_t\eta) \,d\ell\theta(t)\,dP_{x_0}(\eta)\nonumber\\
                 & \ge  \int \min\{1,\dim\tau+\dim\eta\} \,dP_{x_0}(\eta)\nonumber \\
                 & =  \min\{1,\dim\tau+\delta\} \nonumber\\
                 & =  1.\nonumber
\end{align}
But the integrand on the left hand side is $\leq 1$, so it is a.s. equal to $1$, as claimed.
\end{proof}

We can now finish the proof of the theorem.

\begin{proof}[Proof of Theorem \ref{thm:refinement}]
For $\mu$-typical $x$, the analog of Lemma \ref{lem:discrete-equidistribution-and-spectrum} holds for the distributions $P_x$ by assumption. It follows that for $\mu$-a.e. $x$ and any measure $\nu$ for which $x$ equidistributes under $T_\beta$ sub-sequentially, we have a representation similar to Theorem \ref{thm:integral-representation}: \[
  \nu = \int c_\omega \cdot (\delta_{y_\omega}*\eta_\omega)|_{I_\omega}\,dQ(\omega)
\]
where $c_\omega,y_\omega,\eta_\omega, I_\omega$ are defined on some auxiliary probability space $(\Omega,\mathcal{F},Q)$, and $\eta_\omega$ is distributed as $P_x$ (rather than $P$).

The proof is now concluded exactly in the same way as in Theorem \ref{thm:WM-case}. Combining the integral representation with Lemma \ref{lem:refined-dissonance}, for $\mu$-a.e. $x$ and any $\nu$ for which $x$ equidistributes sub-sequentially under $T_\beta$, we have that $\nu$ dissonates with every measure of large enough dimension, and also that $\dim\nu\geq\delta$. But, by Theorem \ref{thm:existence-of-resonant-measures}, this is possible only if $\nu$ is of dimension $1$, hence the unique absolutely continuous measure for $T_\beta$. This completes the proof.
\end{proof}

As for Theorem \ref{thm:refinement-2}, the argument is identical, except that in equation \eqref{eq:refinement} one replaces $\mu$ by $\mu_\omega$ and integrates $dQ(\omega)$. We leave the remaining  details to the reader.

\subsection{\label{prediction-measures}Distributions associated to $T_\gamma$ invariant measures}

Let  $\gamma>1$ and $\mu$ a $T_\gamma$-invariant and ergodic measure with $\dim\mu>0$. In this section we develop some background about such measures and distributions associated to them. This is a minor adaptation of \cite[Section 3]{Hochman12}, which dealt with the integer case (though the language we employ here is slightly different).

Let $G=\lceil \gamma \rceil$. We have already met the $\gamma$-shift $X_\gamma\subseteq[G]^\mathbb{N}$, which, together with the shift map $T$, factors onto $([0,1],T_\gamma)$,  and have noted that $\mu$ lifts uniquely to $X_\gamma$. We also will need the so-called natural extension: let $\widetilde{X}_\gamma\subseteq[G]^\mathbb{Z}$  denote two-sided $\gamma$-shift, i.e. the set of bi-infinite sequences all of whose subwords appear in the one-sided $\gamma$-shift $X_\gamma$. For $\omega\in\widetilde{X}_\gamma$ let $\omega^+=(\omega_1,\omega_2,\ldots)$ and $\omega^-=(\ldots,\omega_{-1},\omega_0)$, and also write $x(\omega)=\pi(\omega^+)$, where $\pi:X_\gamma\to [0,1]$ is the usual base-$\gamma$ coding map. It is a standard fact that $\mu$ lifts uniquely to a $T$-invariant measure $\widetilde{\mu}$ on  $(\widetilde{X}_\gamma,T)$.

For $\widetilde{\mu}$-typical $\omega$, let $\mu_\omega$ denote the conditional measure of $\widetilde{\mu}$ given the ``past'' $(\ldots,\omega_{-1},\omega_{0})$. These conditional measures can be defined abstractly as the disintegration of $\widetilde{\mu}$ given the measurable and countably generated partition into different pasts, see \cite[Theorem 5.14]{EinsiedlerWard11}, or more concretely by the conditions
\[
\mu_\omega[i_1\cdots i_k] = \lim_{n\to\infty}\frac{\mu[\omega_{-n}\ldots \omega_0 i_1\ldots i_k]}{\mu[\omega_{-n}\ldots \omega_0]}.
\]
(That the limit exists for $\widetilde{\mu}$-a.e. $\omega$ can be seen from a martingale argument.)

These conditional measures are measures on the ``future'' $[G]^\mathbb{N}$ and almost surely are supported on the one-sided $\gamma$-shift. We silently shall identify $\mu_\omega$ with the corresponding measure $\pi\mu_\omega$ on $[0,1]$. It is well known that $\dim\mu_\omega=\dim\mu$ a.s.

\begin{defn}\label{def:discrete-quasi-palm}
A distribution $P_0\in\mathcal{D}$ is \emdef{$S_{t_0}$-quasi-Palm} if it is $S_{t_0}$-invariant, gives full mass to $\mathcal{M}^\sqr$, and for every Borel set $B\subseteq\mathcal{M}^\sqr$ with $P(B)=1$ and every $k\in\mathbb{N}$, $P$-almost every measure $\eta$ satisfies $\eta_{x,kt_0}\in B$ for $\eta$-almost all $x$ such that $[x-e^{-{kt_0}},x+e^{kt_0}]\subseteq [-1,1]$.
\end{defn}

If $P_0$ is $S_{t_0}$-quasi-Palm, it is easy to see that $P=\frac{1}{t_0}\int_0^{t_0}S_tP_0\,dt$ is $S$-quasi-Palm.

\begin{defn}\label{def:discrete-efd}
An $S_{t_0}$-invariant and ergodic distribution which is also $S_{t_0}$-quasi-Palm is a \emdef{$t_0$-discrete ergodic fractal distribution}.
\end{defn}

It is again clear that if $P_0$ is such a distribution then $P=\frac{1}{t_0}\int_0^{t_0}S_tP_0\,dt$ is an EFD.

\begin{thm}\label{thm:prediction-measures}
Let $\mu$ be $T_\gamma$-invariant and ergodic. Then there is a $1/\log\gamma$-discrete EFD $P_0$ and a factor map $\sigma : (\widetilde{X}_\gamma,\widetilde{\mu},T)\to (\mathcal{M}^\sqr,P_0,S_{\log\gamma})$, such that for $\widetilde{\mu}$-a.e. $\omega$, \begin{equation}
  (\mu_\omega)_{x(\omega)}\ll \sigma(\omega).  \label{eq:prediction-are-EFD-measures}
\end{equation}
(actually the two measures are proportional on the interval $[-x(\omega),1-x(\omega)]$).
\end{thm}
\begin{proof}
The factor map in question is defined by
\[
  \sigma(\omega)=\lim_{n\to\infty}S_{n\log\gamma}((\mu_{T^{-n}\omega})_{x(T^{-n}\omega)}),
\]
and $P_0$ is the push-forward of $\widetilde{\mu}$ through this map. The $S$-quasi-Palm property is a consequence of the fact that the distribution of $\mu_{T\omega}$ for $\omega\sim\widetilde{\mu}$ is equal in distribution to $\mu_\omega$ for $\omega\sim\widetilde{\mu}$. For a more detailed verification of the integer case, see \cite[Theorem 3.1]{Hochman12}; there are no substantial changes when passing to a general $\gamma>1$.
\end{proof}

Let $P_0$ be as in Theorem \ref{thm:prediction-measures}, and
\[
  P=\frac{1}{\log\gamma}\int_0^{\log\gamma}S_tP_0\,dt
\]
which, as was already noted, is an EFD.

\begin{prop}\label{pro:prediction-measures}
\begin{enumerate}
\item \label{it:pred-measures-1} $P_0$ is $\log\gamma$-generated by $\mu_\omega$ at $\mu_\omega$-a.e. point, for $\widetilde{\mu}$-a.e. $\omega$.
\item \label{it:pred-measures-2} $\mu_\omega$ generates $P$ for $\widetilde{\mu}$-a.e. $\omega$.
\end{enumerate}
\end{prop}
\begin{proof}
\eqref{it:pred-measures-1} By the ergodic theorem, $P_0$-a.e. measure $\nu$  generates $P_0$  $\log\gamma$-discretely at $0$, and by the $S_{\log\gamma}$-quasi-Palm property  of $P_0$, $0$ can be replaced by $\nu$-typical $x$ (this argument is the same as the proof of Lemma \ref{lem:fd-measures-generate-discretely}). Thus  $\sigma(\omega)$ generates $P_0$ $\log\gamma$-discretely, and using \eqref{eq:prediction-are-EFD-measures}, the same is true for $\mu_\omega$.

\eqref{it:pred-measures-2} is a consequence of \eqref{it:pred-measures-1}. Let $f\in C(\mathcal{P}([-1,1]))$ and let
\[
F(\nu)=\frac{1}{\log\gamma}\int_0^{\log\gamma}f(S_t\nu)\,dt.
\]
 Let $\nu=\mu_\omega$ for a typical $\omega$. We must show that \[
  \lim_{T\to\infty}\frac{1}{T}\int_0^Tf(\nu_{x,t})\,dt = \int f\,dP_0 \qquad\text{for }\nu\text{-a.e. }x.
\]
But \[
\frac{1}{T}\int_0^Tf(\nu_{x,t})\,dt = \frac{1}{\lfloor T\rfloor }\sum_{n=0}^{\lfloor T\rfloor}F(\nu_{x,n}) + O\left(\frac{\Vert f\Vert_\infty}{T}\right).
\]
%If $F$ were continuous on $\mathcal{M}^\sqr$, convergence above would follow immediately from the fact that $P_0$ is  $\log\gamma$-discretely generated by $\nu$ at $\nu$-a.e. point. Continuity fails because of the discontinuity of $S_t$. In fact, $F$ is continuous on a set of full $P_0$-measure because $P_0$ is supported on non-atomic measures, but not uniformly so. This can be fixed by a standard approximation argument, which we omit.
If $F$ were continuous on $\mathcal{P}([0,1])$, convergence above would follow immediately from the fact that $P_0$ is  $\log\gamma$-discretely generated by $\nu$ at $\nu$-a.e. point.
In fact, $F$ is defined only on $\mathcal{M}^\sqr$, but it is continuous on a $P_0$-full measure set (such as the set of atomless measures in $\mathcal{M}^\sqr$), so the result still follows, see e.g. \cite[Theorem 2.7]{Billingsley99}.

\end{proof}

\subsection{\label{sub:host-proof}Proof of Theorem \ref{thm:application-HostLindenstrauss}: normality of $\mu$}

Let $\gamma>1$, let $\mu$ be a $T_\gamma$-invariant and ergodic measure, and let $\beta>1$ be a Pisot number with  $\gamma \nsim \beta$. Our goal is to show that $\mu$ is pointwise $\beta$-normal, and so is $f\mu$ when $f\in\diff^2(\mathbb{R})$. We continue with the notation of the previous section: $\widetilde{\mu},\mu_\omega,P_0,P$ etc.

We will first show that $\mu$ is pointwise $\beta$-normal; the case of $f\mu$ for $f\in\diff^2(\R)$ will be handled in the next section.

Suppose first that $\Sigma(P,S)$ does not contain non-zero integer multiples of $1/\log\beta$. Then by Proposition \ref{pro:prediction-measures}\eqref{it:pred-measures-2} and Theorem \ref{thm:WM-pisot-case}, for $\widetilde{\mu}$ a.e. $\omega$ the conditional measure $\mu_\omega$ is pointwise $\beta$-normal. But then so is $\mu$ since $\mu=\int \mu_\omega\, d\widetilde{\mu}(\omega)$.

Therefore, assume that there is some integer $k\neq 0$ with  $k/\log\beta\in\Sigma(P,S)$, or, equivalently, that $P$ is not $S_{\log\beta}$-ergodic. Our goal is to verify that the assumptions of Theorem \ref{thm:refinement-2} are met. In light of Proposition \ref{pro:prediction-measures}, and setting $\Omega=\widetilde{X}_\gamma$ and $Q=\widetilde{\mu}$, we see that all that remains to be checked is that the cumulative phase measure $\theta=\int \theta_{\mu_\omega} \,\widetilde{\mu}(\omega)$ has full dimension.

Recall that  $P$-a.e. $\nu$ equidistributes under $S_{\log\beta}$ for an $S_{\log\beta}$-ergodic distribution (Lemma \ref{lem:discrete-equidistribuion-a.e.}). This is an $S$-invariant property, so it holds for $P_0$-a.e. $\nu$, and by the $S_{\log\lambda}$-quasi-Palm property, the relation  \eqref{eq:prediction-are-EFD-measures}, and Lemma \ref{pro:prediction-measures}, the same holds for $\nu^x$ at $\nu$-a.e. $x$ for $P_0$-a.e. $\nu$.

Fix an eigenfunction $\varphi$ for the eigenvalue $k/\log\beta$. The assumption $\beta\nsim\gamma$ comes in during the proof of the next lemma.

\begin{lem}\label{lem:uniform-phase}
$\theta'=\int\theta_\nu\,dP_0\nu$ is Lebesgue measure on the circle.
\end{lem}
\begin{proof}
  To begin, note that either $P_0=P$ or else $P_0$ is the level set of an eigenfunction $\psi$ with eigenvalue $m/\log\gamma$ for some non-zero integer $m$. In the first case the assertion is clear, since $\theta'$ is a translation-invariant measure on the circle, so we consider the second case only. Switching to additive notation, the map $\nu\mapsto (\varphi(\nu),\psi(\nu))$ defines a factor map from $(P,S)$ to the torus equipped with translation by $(k/\log\beta,m/\log\gamma)$. Since $\beta\nsim\gamma$, Lebesgue measure is the unique invariant measure for this translation, and we deduce that the distribution of $\varphi$ conditioned on any level set of $\psi$ is uniform on the circle; in particular, the distribution of $\varphi$ on $P_0$ is uniform on the circle. Now the lemma follows since, by Proposition \ref{pro:synchronization}, $\theta_\nu=\delta_{\varphi(\nu)}$ for $P$-a.e. $\nu$ and hence, by $S$-invariance, for $P_0$-a.e. $\nu$.
\end{proof}

Since $\mu_\omega\ll\sigma(\omega)$ for $\widetilde{\mu}$-a.e. $\omega$, and $\theta_\nu$ is a single atom, we have $\theta_{\mu_\omega}=\theta_{\sigma(\omega)}$, hence $\theta=\int\theta_{\mu_\omega}\,d\widetilde{\mu}(\omega)$ is uniform on the circle, in particular of dimension $1$. We have shown all the hypotheses of Theorem \ref{thm:refinement-2} hold, so this proves the pointwise $\beta$-normality of $\mu$.

\subsection{\label{sub:host-proof-2}Conclusion of the proof of Theorem \ref{thm:application-HostLindenstrauss}: normality of $f\mu$}

It remains for us to prove pointwise $\beta$-normality of $f\mu$ for $f\in\diff^2(\mathbb{R})$; again we will do so by applying Theorem \ref{thm:refinement-2}. Since $\mu_\omega$ and $f\mu_\omega$ generate and $\log\beta$-generate the same distributions, the task is to show that the cumulative phase measure has dimension $1$.  By Corollary \ref{cor:phase-shift}, this measure is given by
\begin{equation} \label{eq:phase-distribution-cumulative}
   \theta'=\int\int \delta_{e(\log(-f'(x))/\log\gamma)\cdot \varphi(\mu_\omega)} \,d\mu_\omega(x) \,d\widetilde{\mu}(\omega).
\end{equation}

In order to be able to apply projection results, we pass from multiplicative to additive notation; in particular the range of $\varphi$ becomes the unit interval $[0,1]$. Define the measure $\eta$ on $[0,1]^2$ by \[
   \eta = \int \mu_\omega\times\delta_{\varphi(\mu_\omega)} \, d\widetilde{\mu}(\omega)
\]
Then  $\theta'$ is the projection of $\eta$ by the map
\[
   \pi(x,y)=y-\log f'(x)/\log\gamma.
\]

Now, by definition the projection $P_2\eta$ of $\eta$ to the $y$-axis is $\int \delta_{\varphi(\mu_\omega)}d\widetilde{\mu}(\omega)$, which we have seen is Lebesgue measure. Also, since $\dim\mu_\omega=\dim\mu$ a.s., by Lemma \ref{lem:properties-dim-measures}\eqref{it:cavalieri-ineq} we find that
\[
  \dim\eta\geq 1+\dim\mu.
\]

Note that the map $F(x,y)=(x,\pi(x,y))$ preserves dimension: indeed, since $f\in\diff^2(\R)$, we have that $f'$ is differentiable, and one easily computes and finds that $F$ is nonsingular. Thus the image $\widetilde\eta=F\eta$ has dimension $\dim\eta\ge 1+\dim\mu$. Also note that $P_1\widetilde\eta=P_1\eta=\int\mu_\omega\,d\widetilde{\mu}(\omega)=\mu$. Since $\mu$ is exact-dimensional, $\dim_P(P_1\widetilde\eta)=\dim\mu$ (see the discussion at the end of Section \ref{sub:further-dimension}).

On the other hand, $P_2\widetilde\eta=\theta'$ by definition. Thus, applying Lemma \ref{lem:properties-dim-measures}\eqref{it:upper-bound-from-projections} to $\widetilde{\eta}$, we conclude
\begin{align*}
  \dim\theta'       & \geq  \dim\widetilde{\eta} - \dim_P (P_1\widetilde\eta) \\
                    & \geq  \dim\eta - \dim \mu \\
                    & \geq  1.
\end{align*}

Summarizing, we have shown that $\dim\theta'=1$, hence we can apply Theorem \ref{thm:refinement-2} to $f\mu$. This completes the proof of Theorem \ref{thm:application-HostLindenstrauss}.

\subsection{\label{sub:analytic-image-of-ssms}Proof of Theorem \ref{thm:analytic-images-of-ssms}}

We now assume that $\mathcal{I}$ consists of linear maps, $\mu$ is self-similar, and  $f\in\diff^\omega(\mathbb{R})$ is not affine. Our aim is to prove Theorem \ref{thm:analytic-images-of-ssms}, asserting the $\beta$-normality of $f\mu$ for all Pisot $\beta>1$. The argument is similar to what we have already seen, except that the classical projection theorems are not strong enough and we rely instead on a recent result from \cite{Hochman13} that gives stronger bounds for self-similar measures.

If $\mathcal{I}$ contains two maps with contraction ratios $\lambda_1\nsim \lambda_2$, then it follows from Theorem \ref{thm:dissonant-IFSs} that $f\mu$ is normal to all Pisot bases for all $f\in\diff^2(\R)$ (and in fact $f\in\diff^1(\R)$ is enough in this case). Thus we may assume $\lambda(f_i)\sim\gamma^{p_i}$ for all $f_i\in\mathcal{I}$ and integers $p_i$. Let $\beta>1$ be Pisot. Again, if $\gamma\nsim\beta$ then we are done by Theorem \ref{thm:dissonant-IFSs}, so we assume that $\beta\sim\gamma$.

\begin{lem}
$\mu$ generates an ergodic distribution $P$ with $\Sigma(P,S)\subseteq (1/\log\gamma)\mathbb{Q}$. For each eigenvalue, the phase measure is well defined and consists of a single atom.
\end{lem}
\begin{proof}
The first statement follows from the fact that the measure in question is a quasi-product measure, and from Theorems \ref{thm:gibbs-generate-EFDs} and \ref{thm:spectrum-of-quasi-product}. Because the contractions are linear, it is elementary to see that for every accumulation point $\nu$ of $\mu_{x,t}$, there is a linear map $f$ and interval $I$ with $f\mu=c\cdot\nu|_I$ for a normalizing constant $c$. The statement about the phase then follows from Proposition \ref{pro:synchronization}.
\end{proof}

Let $P$ be as in the lemma, and fix an eigenvalue $\alpha$ of $P$ with associated eigenfunction $\varphi$. Let $f\in\diff^\omega(\mathbb{R})$ be non-linear. By Corollary \ref{cor:phase-shift}, we know that the phase distribution of $f\mu$ is, up smooth coordinate change and in additive notation, the push-forward of $\mu$ through $f'$. Since $f$ is real analytic and non-linear, $f'$ is a piecewise diffeomorphism and so $\dim f'\mu = \dim\mu>0$.

Now let $\tau_n$ be the sequence of eventually-resonant measures for $T_\beta$-invariant measures provided by Theorem \ref{thm:existence-of-resonant-measures}, and observe from the construction of $\tau_n$ that they are in fact affine combinations of self-similar measures with uniform contraction ratio $\sim\beta$ satisfying the open set condition. Arguing through the proof of Theorem \ref{thm:refinement}, we find that the following lemma, which replaces Lemma \ref{lem:refined-dissonance}, allows the proof to carry through.

\begin{lem}
Let $\nu=f\mu$ and $\theta=\theta_{\nu}$. Let $P_x$ denote the distribution that is $\log\gamma$-generated by $\nu$ at $x$. Let $\tau$ be a self-similar measure for an IFS with uniform contraction ratio a power of $\gamma$, satisfying the open set condition, and satisfying $\dim\tau+\dim\nu\geq 1$. Then for $\nu$-a.e. $x$ and $P_x$-a.e. $\eta$, we have $\dim\tau*\eta = \min\{1,\dim\tau+\dim\eta\}$.
\end{lem}
\begin{proof}
We have already noted that  $\dim\theta>0$. We now calculate exactly as in the proof of Lemma \ref{lem:refined-dissonance}. The only change is that we cannot use Theorem \ref{thm:marstrand-dimofexceptions} to deduce that $\dim \tau*S_{\ell(z)}\eta=1$ for $\theta$-almost all $z$, since all we know about $\theta$ is that $\dim\theta>0$.

Instead, note that up to a translation, $S_t\eta$ is absolutely continuous with respect to the measure $\eta$ scaled by $e^{-t}$, whence $\tau*S_t\eta$ is absolutely continuous with respect to the image of the self-similar measure $\tau\times\mu$ via the linear map $(x,y)\mapsto x+e^{-t}y$. Now, for $\tau=\tau_n$, both $\mu$ and $\tau$ are self-similar with contraction ratios $\sim\beta$. When the contraction ratios of $\mathcal{I}$ are uniform, $\tau\times\mu$ is also self-similar and of dimension $>1$ (for large $n$), and we can invoke Theorem \cite[Theorem 1.8]{Hochman13}, which implies that $\dim \tau*S_t\eta=1$ for all $t$ outside a set of Hausdorff dimension $0$, and hence of $\theta$-measure $0$. In the non-uniformly contracting case a minor (but not short) modification of the arguments in \cite{Hochman13} is needed; this will appear separately.
\end{proof}

This completes the proof of Theorem \ref{thm:analytic-images-of-ssms}.

\bibliographystyle{plain}
\bibliography{bib}

\medskip
\noindent Email: mhochman@math.huji.ac.il \\
Address: Einstein Institute of Mathematics, Givat Ram, Jerusalem 91904, Israel

\medskip
\noindent Email: pshmerkin@utdt.edu \\
Address: Department of Mathematics, Faculty of Engineering and Physical Sciences, University of Surrey, Guildford, GU2 7XH, United Kingdom \\
Current address: Department of Mathematics and Statistics, Torcuato Di Tella University, Av. Figueroa Alcorta 7350 (1425), Buenos Aires, Argentina.

\end{document}